\documentclass{amsart}
\usepackage{stmaryrd}
\usepackage{silence}
\usepackage{upgreek}
\usepackage{amssymb}
\usepackage{faktor}
\usepackage{mathrsfs}
\usepackage{tikz}
\usepackage{tikz-cd}
\usepackage{mathtools}
\usepackage[mathcal]{euscript}
\usepackage{graphicx}
\usepackage{url}
\usepackage{hyperref}
\usepackage[T1]{fontenc}
\usepackage{subfigure}

\usetikzlibrary{shapes.geometric}
\usetikzlibrary{decorations.markings}
\usetikzlibrary{patterns,decorations.pathreplacing}

\hypersetup{colorlinks=true,linkcolor=blue}

\usepackage{ragged2e}

\newcommand{\intrinsic}[1]{S(#1)}

\newcommand{\fullyreduced}[1]{\widetilde{S}\left(#1\right)}

\newcommand{\sgn}{\mathrm{sgn}}

\definecolor{coloryellow}{RGB}{240,228,66}
\definecolor{colorskyblue}{RGB}{86,180,233}
\definecolor{colorvermillion}{RGB}{213,94,0}

\newcommand{\graphfont}{\mathsf}

\newcommand{\thetagraph}[1]{\graphfont{\Theta}_{#1}}
\newcommand{\completegraph}[1]{\graphfont{K}_{#1}}

\newcommand{\stargraph}[1]{\graphfont{S}_{#1}}
\newcommand{\graf}{\graphfont{\Gamma}}

\newcommand{\lollipopgraph}[1]{\graphfont{L}_{#1}}

\newcommand{\angelgraph}{\graphfont{\Omega}}

\DeclareSymbolFont{sfletters}{OT1}{cmss}{m}{n}
\DeclareMathSymbol{\sTheta}{\mathord}{sfletters}{"02}

\theoremstyle{definition}
\newtheorem{definition}{Definition}[section]

\newtheorem{notation}[definition]{Notation}

\newtheorem{example}[definition]{Example}

\newtheorem{construction}[definition]{Construction}
\newtheorem{observation}[definition]{Observation}

\theoremstyle{plain}
\newtheorem{proposition}[definition]{Proposition}
\newtheorem{lemma}[definition]{Lemma}
\newtheorem{corollary}[definition]{Corollary}
\newtheorem{theorem}[definition]{Theorem}
\newtheorem{conjecture}[definition]{Conjecture}

\theoremstyle{remark}
\newtheorem{remark}[definition]{Remark}

\makeatletter
\@addtoreset{definition}{section}
\makeatother

\usepackage{marginnote}
    \DeclareFontFamily{U}{wncy}{}
    \DeclareFontShape{U}{wncy}{m}{n}{<->wncyr10}{}
    \DeclareSymbolFont{mcy}{U}{wncy}{m}{n}
    \DeclareMathSymbol{\Sha}{\mathord}{mcy}{"58}

\newsavebox{\foobox}

\title{On the second homology of planar graph braid groups}

\author{Byung Hee An}
\email{anbyhee@knu.ac.kr}
\address{Department of Mathematics Education, Teachers College, Kyungpook National University, Daegu, South Korea}
\author{Ben Knudsen}
\email{b.knudsen@northeastern.edu}
\address{Department of Mathematics, Northeastern University, Boston, MA 02115, USA}

\begin{document}
\begin{abstract}
We show that the second homology of the configuration spaces of a planar graph is generated under the operations of embedding, disjoint union, and edge stabilization by three atomic graphs: the cycle graph with one edge, the star graph with three edges, and the theta graph with four edges. We give an example of a non-planar graph for which this statement is false.
\end{abstract}

\maketitle

\section{Introduction}

A fundamental dichotomy, observed since the dawn of the study of configuration spaces, is that of stability and instability \cite{Arnold:CRGDB,McDuff:CSPNP,Church:HSCSM,ChurchEllenbergFarb:FIMSRSG}. Writing $B_k(X)$ for the $k$th unordered configuration space of the topological space $X$, a phenomenon is said to be \emph{stable} if it occurs for all sufficiently large $k$, or perhaps in the limit as $k\to \infty$. Stable phenomena tend to be structured and calculable, unstable phenomena fleeting and irregular, hence more difficult to grasp.

This paper is concerned with the unstable homology of configuration spaces of graphs, or equivalently of their fundamental groups, the graph braid groups \cite{Abrams:CSBGG}. This investigation is a companion and counterpoint to the recent complete calculation of the stable homology \cite{AnDrummondColeKnudsen:AHGBG}, premised on the action by \emph{edge stabilization} of the polynomial ring generated by the edges of the graph $\graf$ on $H_*(B(\graf))$, where $B(\graf):=\bigsqcup_{k\geq0}B_k(\graf)$ \cite{AnDrummondColeKnudsen:ESHGBG}. 

Unstably, little systematic is known beyond the landmark calculation by Ko--Park of the first homology \cite{KoPark:CGBG}. One consequence of this calculation is that $H_1(B(\graf))$ is generated under edge stabilization by \emph{loop classes} and \emph{star classes}, geometric generators represented by maps from circles. Disjoint unions of stars and loops then give rise to higher degree classes represented by maps from tori.

As observed independently in \cite{ChettihLuetgehetmann:HCSTL} and \cite{WiltshireGordon:MCSSC}, the space $B_3(\thetagraph{4})$ has the homotopy type of a surface of genus $3$, whose fundamental class cannot be represented by a map from a torus---here, $\thetagraph{4}$ is the suspension of four points. Our main result is that, in the planar case, this \emph{theta class} is the only ``exotic'' generator in degree $2$.

\begin{theorem}\label{thm:main}
Let $\graf$ be a planar graph with set of edges $E$. The $\mathbb{Z}[E]$-module $H_2(B(\graf))$ is generated by toric classes and theta classes.
\end{theorem}

Precise descriptions of the action of $\mathbb{Z}[E]$ and of the classes in question can be found in Sections \ref{section:preliminaries} and \ref{section:generators and relations}, respectively. As illustrated by Example \ref{example:non-planar}, the assumption of planarity cannot be removed.

\subsection{Questions} Our work invites the following questions, which we hope to pursue in the future.\\

\begin{enumerate}
\item \emph{Universal relations}. Star classes and loop classes are universal generators for the first homology of graph braid groups, and Theorem \ref{thm:main} provides universal generators for the second homology in the planar case. In degree $1$, a complete set of universal relations is known. What are the relations in degree $2$?\\
\item \emph{Combinatorial bases}. Ko--Park give a basis for $H_1(B(\graf))$ in terms of combinatorial invariants of $\graf$ \cite[Thm. 3.16]{KoPark:CGBG}. Can a similar basis be given for $H_2(B(\graf))$?\\
\item \emph{Higher degrees}. We conjecture that a direct analogue of Theorem \ref{thm:main} holds in all degrees.
\begin{conjecture}\label{conj:main}
For any planar graph $\graf$ and $i>0$, the $\mathbb{Z}[E]$-module $H_i(B(\graf))$ is generated by classes arising from disjoint unions of cycle graphs, star graphs, and theta graphs.
\end{conjecture}
\item \emph{Non-planar graphs}. What extra generators are needed in order to remove the assumption of planarity from Theorem \ref{thm:main} and Conjecture \ref{conj:main}?
\end{enumerate}

\subsection{Strategy and outline} Writing $M(\graf)$ for the submodule spanned by toric classes and theta classes, the theorem is the equality $M(\graf)=H_2(B(\graf))$. This conclusion being well known for trees, we proceed by induction on the first Betti number of $\graf$. The tool facilitating this induction is the exact sequence of Proposition \ref{prop:vertex explosion}, originally introduced in \cite{AnDrummondColeKnudsen:SSGBG}, which expresses $H_2(B(\graf))$ as an extension by a submodule contained in $M(\graf)$ by induction. By exactness, proving the theorem becomes a matter of showing that $M(\graf)$ surjects onto the quotient by this submodule (Theorem \ref{thm:reformulation}), which is achieved through consideration of a special generating set for the cokernel (Proposition \ref{prop:cokernel generators}). Two inductions reduce the vanishing of this generating set to the triconnected case, which is handled by a combinatorial argument.

In linear order, Section \ref{section:preliminaries} is concerned with background material, Section \ref{section:generators and relations} discusses the geometric classes of interest, Section \ref{section:decomposing graphs} details decomposition tactics for the purposes of induction, Section \ref{section:pesky cycles} provides the key reformulation of Theorem \ref{thm:main} and the special generating set for the cokernel, Section \ref{section:examples and a first reduction} is concerned with examples and the first induction, Section \ref{section:base cases} deals with the triconnected case, and Section \ref{section:induction} carries out the second induction.

\subsection{Acknowledgements} This paper benefited substantially from conversations with Gabriel Drummond-Cole, who declined to be named as coauthor. The authors also thank Roberto Pagaria for a simplifying observation. The first author was supported by the National Research Foundation of Korea (NRF) grant funded by the Korea government (MSIT) (No. 2020R1A2C1A0100320). The second author was supported by NSF grant DMS 1906174.

\section{Preliminaries}\label{section:preliminaries}

This section presents a brief overview of some necessary background material in the study of the homology of graph braid groups. A more leisurely exposition is available in \cite{AnDrummondColeKnudsen:ESHGBG}.

\subsection{Conventions on graphs}

A graph is a finite CW complex of dimension at most $1$, whose $0$-cells and open $1$-cells are called vertices and edges, respectively. A graph is called a tree if it is contractible and a cycle if it is homeomorphic to $S^1$. A half-edge is a point in the preimage of a vertex under the attaching map of a $1$-cell;
thus, every edge determines two half-edges. In general, sets of vertices, edges, and half-edges are denoted $V(\graf)$, $E(\graf)$, and $H(\graf)$, respectively, but we omit $\graf$ from the notation wherever doing so causes no ambiguity.

A half-edge $h$ has an associated vertex $v(h)$ and an associated edge $e(h)$, and we write $H(v)=\{h\in H: v=v(h)\}$ for the set of half-edges incident on $v\in V$. The degree or valence of $v$ is $d(v)=|H(v)|$. A vertex is essential if its valence is at least $3$. An edge is a tail if its closure contains a vertex of valence $1$ and a self-loop if its closure contains only one vertex. A multiple edge is a set of edges incident on the same pair of vertices. 

A subgraph is a subcomplex of a graph. A graph morphism is a finite composition of isomorphisms onto subcomplexes and inverse subdivisions, which we call smoothings---see Figure \ref{figure:smoothing} and \cite[\S2.1]{AnDrummondColeKnudsen:ESHGBG}.

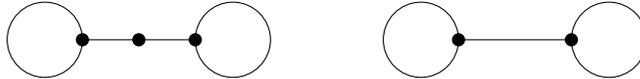
\begin{figure}[ht]
\begin{center}
\begin{tikzpicture}
\fill[black] (0,0) circle (2.5pt);
\fill[black] (1.5,0) circle (2.5pt);
\fill[black] (.75,0) circle (2.5pt);
\draw(0,0) -- (1.5,0);
\draw(-.5,0) circle (.5cm);
\draw(2,0) circle (.5cm);
\begin{scope}[xshift=5cm]
\fill[black] (0,0) circle (2.5pt);
\fill[black] (1.5,0) circle (2.5pt);
\draw(0,0) -- (1.5,0);
\draw(-.5,0) circle (.5cm);
\draw(2,0) circle (.5cm);
\end{scope}
\end{tikzpicture}
\end{center}
\caption{Two graph structures on a pair of handcuffs, for which the identity is a smoothing from left to right}\label{figure:smoothing}
\end{figure}

We close with several important families of graphs, examples of which are depicted in Figure \ref{fig:graph examples}.

\begin{example}
The \emph{star graph} $\stargraph{n}$ is the cone on the discrete space $\{1,\ldots, n\}$.
\end{example}

\begin{example}
The \emph{theta graph} $\thetagraph{n}$ is the unique graph with two vertices, $n$ edges, and no self-loops.
\end{example}

\begin{example}
The \emph{complete bipartite graph} $\completegraph{m,n}$ is the join of the discrete spaces $\{1,\ldots, m\}$ and $\{1,\ldots, n\}$.
\end{example}

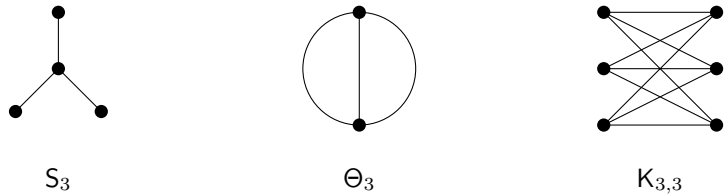
\begin{figure}[ht]
\begin{center}
\begin{tikzpicture}
\begin{scope}
\fill[black] (0,0) circle (2.5pt);
\fill[black] (0,.75) circle (2.5pt);
\fill[black] (.57,-.57) circle (2.5pt);
\fill[black] (-.57,-.57) circle (2.5pt);
\draw(0,0) -- (0,.75);
\draw(0,0) --(.57,-.57);
\draw(0,0) --(-.57,-.57);
\draw(0,-1.2) node[below]{$\stargraph{3}$};
\end{scope}
\begin{scope}[xshift=4 cm]
\fill[black] (0,-.75) circle (2.5pt);
\fill[black] (0,.75) circle (2.5pt);
\draw(0,0) circle (.75cm);
\draw(0,.75) -- (0,-.75);
\draw(0,-1.2) node[below]{$\thetagraph{3}$};
\end{scope}
\begin{scope}[xshift=8cm]
\foreach  \x in {-.75,.75}\foreach \y in {-.75,0,.75}
\fill[black] (\x,\y) circle (2.5pt);
\draw (-.75,.75) -- (.75,.75) -- (-.75,0) -- (.75,0) -- (-.75, -.75) -- (.75, .75);
\draw (.75, -.75) -- (-.75,.75) -- (.75, 0);
\draw (-.75,0) -- (.75,-.75)-- (-.75,-.75);
\draw(0,-1.2) node[below]{$\completegraph{3,3}$};
\end{scope}
\end{tikzpicture}
\end{center}
\caption{Examples of graphs}\label{fig:graph examples}
\end{figure}

\subsection{The \'{S}wi\k{a}tkowski complex}\label{section:swiatkowski} Our object of study in this paper is the homology of the configuration spaces of a graph $\graf$. Our primary weapon is a certain chain complex, which we now define---see \cite[\S2.2]{AnDrummondColeKnudsen:ESHGBG} for further discussion.

\begin{definition}
Let $\graf$ be a graph. For $v\in V$, write $S(v)=\mathbb{Z}\langle\varnothing, v, h\in H(v)\rangle$. The \emph{\'{S}wi\k{a}tkowski complex} is the $\mathbb{Z}[E]$-module \[\intrinsic{\graf}=\mathbb{Z}[E]\otimes\bigotimes_{v\in V} S(v),\] equipped with the bigrading $|\varnothing|=(0,0)$, $|v|=|e|=(0,1)$, and $|h|=(1,1)$, together with the differential determined by the equation $\partial(h)=e(h)-v(h)$.
\end{definition}

Note that the differential $\partial$ preseves the second grading, which corresponds to the number of particles in a configuration. We refer to this auxiliary grading as \emph{weight}.

We systematically omit all factors of $\varnothing$ and all tensor symbols when dealing with elements of $\intrinsic{\graf}$, and we regard half-edge generators at different vertices as permutable up to sign. 

\begin{theorem}[{\cite[Thm. 2.10]{AnDrummondColeKnudsen:ESHGBG}}]
There is a natural isomorphism of bigraded $\mathbb{Z}[E]$-modules \[H_*(B(\graf))\cong H_*(\intrinsic{\graf}).\]
\end{theorem}

Several comments are in order. First, precursors to this result can be found in \cite{Swiatkowski:EHDCSG} and \cite{ChettihLuetgehetmann:HCSTL}. Second, the action of $\mathbb{Z}[E]$ on the lefthand side arises from an $E$-indexed family of \emph{edge stabilization maps}. Stabilization at $e$ replaces the subconfiguration of particles lying in the closure of $e$ with the collection of averages of consecutive particles and endpoints---see Figure \ref{fig:edge stabilization} and \cite[\S2.2]{AnDrummondColeKnudsen:ESHGBG}. Third, regarding the implied functoriality, we direct the reader to \cite[\S2.3]{AnDrummondColeKnudsen:ESHGBG}.

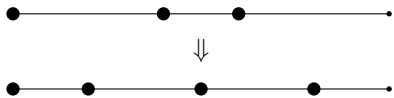
\begin{figure}[ht]
\begin{tikzpicture}
\fill[black] (0,0) circle (2.5pt);
\fill[black] (5,0) circle (1pt);
\fill[black] (2,0) circle (2.5pt);
\fill[black] (3,0) circle (2.5pt);
\draw(0,0) -- (5,0);
\draw(2.5,-.2) node[below]{$\Downarrow$};
\begin{scope}[yshift=-1cm]
\fill[black] (0,0) circle (2.5pt);
\fill[black] (5,0) circle (1pt);
\fill[black] (1,0) circle (2.5pt);
\fill[black] (2.5,0) circle (2.5pt);
\fill[black] (4,0) circle (2.5pt);
\draw(0,0) -- (5,0);
\end{scope}
\end{tikzpicture}
\caption{Edge stabilization}\label{fig:edge stabilization}
\end{figure}

\subsection{Reduction and explosion}\label{section:vertex explosion}

The \emph{reduced} \'{S}wi\k{a}tkowski complex $\fullyreduced{\graf}$ is obtained by replacing $S(v)$ in the definition of $\intrinsic{\graf}$ with the submodule $\fullyreduced{v}\subseteq S(v)$ spanned by $\varnothing$ and all differences of half-edges. The inclusion $\fullyreduced{\graf}\subseteq \intrinsic{\graf}$ is a quasi-isomorphism as long as $\graf$ has no isolated vertices \cite[Prop. 4.9]{AnDrummondColeKnudsen:SSGBG}. Note that, for any $h_0\in H(v)$, a basis for $\fullyreduced{v}$ is given by $\{\varnothing\}\cup \{h-h_0\}_{h_0\neq h\in H(v)}$. In this way, a (non-canonical) basis for $\fullyreduced{\graf}$ may be obtained.

Given a graph $\graf$ and $v\in V$, we write $\graf_v$ for the graph obtained by exploding the vertex $v$---see Figure \ref{fig:vertex explosion} and \cite[Def. 2.12]{AnDrummondColeKnudsen:ESHGBG}---which we regard as a subgraph of a subdivision of $\graf$, uniquely up to isotopy. More generally, given a subset $W\subseteq V$, we write $\graf_W$ for the graph obtained by exploding each of the vertices in $W$. \\

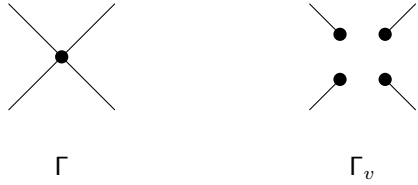
\begin{figure}[ht]
\begin{center}
\begin{tikzpicture}
\begin{scope}
\fill[black] (0,0) circle (2.5pt);
\draw(-.707,-.707) -- (.707,.707);
\draw(-.707,.707) -- (.707,-.707);
\draw(0,-1.2) node[below]{$\graf$};
\end{scope}
\begin{scope}[xshift=4cm]
\fill[black] (-.3,-.3) circle (2.5pt);
\fill[black] (-.3,.3) circle (2.5pt);
\fill[black] (.3,-.3) circle (2.5pt);
\fill[black] (.3,.3) circle (2.5pt);
\draw(-.707,-.707) -- (-.3,-.3);
\draw(.707,.707) -- (.3,.3);
\draw(-.707,.707) -- (-.3,.3);
\draw(.707,-.707) -- (.3,-.3);
\draw(0,-1.2) node[below]{$\graf_v$};
\end{scope}
\end{tikzpicture}
\end{center}
\caption{A local picture of vertex explosion}
\label{fig:vertex explosion}
\end{figure}

\begin{proposition}[{\cite[Prop. 2.3]{AnDrummondColeKnudsen:ESHGBG}}]\label{prop:vertex explosion}
Fix $v\in V$ and $h_0\in H(v)$. The sequence {\small\[\cdots\to H_i(B_k(\graf_v))\xrightarrow{\iota_*} H_i(B_k(\graf))\xrightarrow{\psi} \bigoplus_{h_0\neq h\in H(v)}H_{i-1}(B_{k-1}(\graf_v))\xrightarrow{\delta} H_{i-1}(B_k(\graf_v))\to\cdots\]} is exact. Here, 
\begin{enumerate}
\item the map $\iota_*$ is induced by the inclusion of $\graf_v$,
\item the map $\psi$ is induced by the chain map on reduced \'{S}wi\k{a}tkowski complexes sending $b+\sum(h-h_0)a_h$ to $(a_h)$, where $b$ involves no half-edge generators at $v$, and
\item on the summand indexed by $h$, the map $\delta$ is multiplication by the ring element $e(h)-e(h_0)$.
\end{enumerate} All maps shown are natural and compatible with edge stabilization.
\end{proposition}

The same exact sequence obtains with homology replaced by homology with coefficients in an arbitrary commtuative ring.

We close this section by drawing the following simple, but useful, consequence.

\begin{lemma}\label{lem:1-cut torsion}
Let $v$ be a vertex of $\graf$ and $e_1$ and $e_2$ edges lying in distinct components of $\graf_v$. The $(e_1-e_2)$-torsion submodule of $H_1(B(\graf))$ is contained in the image of $H_1(B(\graf_v))$. 
\end{lemma}
\begin{proof}
Our assumption implies that multiplication by $e_1-e_2$ is injective on the third term in the exact sequence \[\cdots\to H_1(B_k(\graf_v))\to H_1(B_k(\graf))\to  \bigoplus_{d(v)-1}H_0(B_{k-1}(\graf_v))\to \cdots,\] and the claim follows.
\end{proof}

\section{Generators and relations}\label{section:generators and relations}

This section introduces loop, star, and theta classes, the atomic homology classes involved in Theorem \ref{thm:main}, and explores some relations among them.

\subsection{Loop classes and star classes} We begin with two basic types of class in $H_1(B(\graf))$. The reader is directed to \cite[\S5.1]{AnDrummondColeKnudsen:SSGBG} for further details.

\begin{example}\label{example:loop class}
Since $\graf=B_1(\graf)$ is a subspace of $B(\graf)$, an oriented cycle in $\graf$ determines an element of $H_*(B(\graf))$, called a \emph{loop class}. We denote loop classes generically by the letter $\beta$. 
\end{example}

A standard chain level representative of a loop class is obtained by summing the differences of half-edges involved in the cycle in question. For example, the standard representative of the unique loop class in the graph $\lollipopgraph{}$ depicted in Figure \ref{fig:lollipop}, oriented clockwise, is $b=h-h'\in \fullyreduced{\lollipopgraph{}}$.

\begin{figure}[ht]
\begin{tikzpicture}[scale=.75]
\fill[black] (0,-1) circle (10/3pt);
\fill[black] (0,-2.5) circle (10/3pt);
\draw(0,0) circle (1cm);
\draw(0,-2.5) -- (0,-1);
\draw (.35,-.65) node{$h'$}; 
\draw (-.27,-.65) node{$h$}; 
\draw (0,1.25) node{$e'$}; 
\draw (0.2,-1.75) node{$e$}; 
\end{tikzpicture}
\caption{The lollipop graph $\lollipopgraph{}$}
\label{fig:lollipop}
\end{figure}
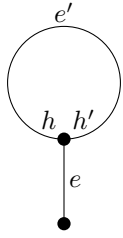

\begin{example}\label{example:star class}
In view of the homotopy equivalence $S^1\simeq B_2(\stargraph{3})$, the choice of half-edges $h_1$, $h_2$, and $h_3$ sharing a common vertex determines a \emph{star class} in $H_1(B_2(\graf))$, which depends on the ordering only up to sign. We denote star classes by $\alpha$ or, e.g., $\alpha_{123}$ if we wish to emphasize the particular choice of half-edges.
\end{example}

Writing $e_j$ for the edge associated to $h_j$, a standard chain level representative for a star class is given by the sum \[a=e_3(h_1-h_2)+e_2(h_3-h_1)+e_1(h_2-h_3).\] The alternative expression \[a=(e_1-e_3)(h_2-h_1)-(e_1-e_2)(h_3-h_1)\] in the basis for $\fullyreduced{\graf}$ privileging $h_1$ is also useful.

In what follows, we refer to the standard representatives introduced above as \emph{loop cycles} and \emph{star cycles} respectively.

\begin{definition}\label{def:support}
The \emph{support} of a star cycle $a$ is the vertex $v(h)$, where $h$ is any half-edge involved in $a$. The \emph{support} of a loop cycle $b$ is the union of the edges $e(h)$ and vertices $v(h)$, where $h$ ranges over all half-edges involved in $b$.
\end{definition}

In other words, the support of a star cycle is the essential vertex used in its definition, and the support of a loop cycle is the loop used in its definition. The reader is warned that support is not well-defined at the level of homology.

\begin{proposition}[{\cite[Prop. 5.6]{AnDrummondColeKnudsen:SSGBG}}]\label{prop:stars and loops generate}
If $\graf$ is connected, then $H_1(B(\graf))$ is generated over $\mathbb{Z}[E]$ by star classes and loop classes.
\end{proposition}

Star classes and loop classes interact according to a relation called the \emph{Q-relation}. Our notation will refer to the graph $\lollipopgraph{}$ of Figure \ref{fig:lollipop}, but functoriality propagates the relation to any graph with a subgraph isomorphic to a subdivision of $\lollipopgraph{}$. Writing $\beta$ for the clockwise oriented loop class and $\alpha$ for the counterclockwise oriented star class in $\lollipopgraph{}$, we have the following.

\begin{lemma}[Q-relation]\label{lem:Q}
In the homology of the configuration spaces of the graph $\lollipopgraph{}$, there is the relation
$(e-e')\beta=\alpha$.
\end{lemma}

Star classes at distinct vertices can also be related. Referring to the graph $\thetagraph{3}$ as shown in Figure \ref{fig:graph examples}, and writing $\alpha$ and $\alpha'$ for the clockwise oriented star classes at the top vertex and bottom vertices, respectively, we have the following relation.

\begin{lemma}[$\theta$-relation]\label{lem:theta}
In the homology of the configuration spaces of the graph $\thetagraph{3}$, there is the relation of star classes
$\alpha-\alpha'=0$.
\end{lemma}

This last relation has the following amusing consequence, which is left as an exercise (or see \cite[p. 60]{AnDrummondColeKnudsen:SSGBG}).

\begin{example}\label{example:K33 stars}
Any two star classes in $H_1(B_2(\completegraph{3,3}))$ are equal, regardless of orientation. In particular, any such star class in is $2$-torsion.
\end{example}

It is useful to distinguish those star classes involved in no instances of the $\theta$-relation.

\begin{definition}
A star cycle with support $v$ is \emph{rigid} if it involves half-edges lying in multiple components of $\graf_v$. A star class is \emph{rigid} if it has a rigid representative.
\end{definition}

According to \cite[Lem. 3.15]{AnDrummondColeKnudsen:ESHGBG}, any star cycle representing a rigid star class is rigid.

In view of the homeomorphism $B(\graf_1\sqcup \graf_2)\cong B(\graf_1)\times B(\graf_2)$, classes in the homology of $B(\graf)$ represented by cycles in the configuration spaces of disjoint subgraphs give rise to an \emph{external product} class in $H_*(B(\graf))$ \cite[Def. 5.10]{AnDrummondColeKnudsen:SSGBG}. At the level of \'{S}wi\k{a}tkowski complexes, the external product is represented by the tensor product of representing cycles. The reader is cautioned that the external product may depend on the choice of representing cycles.

Note that stabilizations of external products of loop classes and star classes are represented by maps from tori. 

\subsection{Theta classes} In this section, we give an elementary description of the non-toric class in $H_2(B_3(\thetagraph{4}))$ discovered in \cite{ChettihLuetgehetmann:HCSTL} and \cite{WiltshireGordon:MCSSC} terms of the combinatorics of the \'{S}wi\k{a}tkowski complex. We begin with a simple lemma.

\begin{lemma}\label{lem:star} In the following, $2\leq i\leq 4$ and $2\leq j<k\leq 4$. \begin{enumerate}
\item The Abelian group $H_1(B_2(\stargraph{4}))$ is freely generated by the classes $\alpha_{1jk}$.
\item The Abelian group $H_1(B_3(\stargraph{4}))$ is generated by the classes $(e_i-e_1)\alpha_{1jk}$ and $e_1\alpha_{1jk}$, subject only to the relation \[(e_4-e_1)\alpha_{123}-(e_3-e_1)\alpha_{124}+(e_2-e_1)\alpha_{134}=0.\]
\end{enumerate}
\end{lemma}
\begin{proof}
Since both groups are spanned by stabilized star classes by Proposition \ref{prop:stars and loops generate}, the relations of \cite[Lem. 2.9, 2.10]{AnDrummondColeKnudsen:AHGBG} imply generation and the validity of the relation. Since $H_1(B_2(\stargraph{4}))$ and $H_1(B_3(\stargraph{4}))$ are free Abelian with respective ranks $3$ and $11$ by \cite[Cor. 4.2]{FarleySabalka:DMTGBG}, the claim follows.
\end{proof}

\begin{proposition}\label{prop:theta class}
The group $H_2(B_3(\thetagraph{4}))$ is free Abelian of rank $1$.
\end{proposition}
\begin{proof}
Consider the exact sequence \[0\to H_2(B_3(\thetagraph{4}))\xrightarrow{\psi} H_1(B_2(\stargraph{4}))^{\oplus 3}\xrightarrow{\delta} H_1(B_3(\stargraph{4}))\] arising from Proposition \ref{prop:vertex explosion} after exploding one of the vertices of $\thetagraph{4}$, privileging the half-edge incident on $e_1$. It follows from Lemma \ref{lem:star} and the explicit formula for $\delta$ given in Proposition \ref{prop:vertex explosion} that the image of $\delta$ is free Abelian of rank $8$, implying the claim.
\end{proof}

\begin{definition}\label{def:theta class}
Let $\graf$ be a graph. A \emph{theta class} is a class in $H_2(B_3(\graf))$ that is the image of a generator of $H_2(B_3(\thetagraph{4}))$ under a topological embedding $\thetagraph{4}\to \graf$.
\end{definition}

In what follows, it will be useful to have a second method of accessing theta classes. We subdivide one of the edges of $\thetagraph{4}$ once by adding a bivalent vertex $v$ and write (abusively) $\thetagraph{v}$ for the graph obtained by exploding this bivalent vertex.

\begin{lemma}\label{lem:theta les}
In the long exact sequence for the vertex explosion $\thetagraph{v}$, the map $\psi$ sends a generator of $H_2(B_3(\thetagraph{4}))$ to the unique (up to sign) non-rigid star class in $H_1(B_2(\thetagraph{v}))$.
\end{lemma}
\begin{proof}
Denoting the star and theta classes in question by $\alpha$ and $\tau$, respectively, the $\theta$-relation implies that $\alpha\in\ker(\delta)$. Thus, by exactness and Proposition \ref{prop:theta class}, we conclude that $\alpha=n\psi(\tau)$ for some $0\neq n\in\mathbb{Z}$. The exact sequence in question is valid (and $\alpha$ nonzero) over any field, and carrying out the same argument over $\mathbb{F}_p$ shows that $p\nmid n$ for every prime $p$, whence $n=\pm 1$.
\end{proof}

For the sake of completeness, we now give an explicit chain level representative for the generator of $H_2(B_3(\thetagraph{4}))$, although we will not have cause to use it.

\begin{construction}\label{construction:theta class} Denote the vertices of $\thetagraph{4}$ as $v_1$ and $v_2$, and write $h_{i,j}$ for the half-edge incident on $v_i$ and $e_j$. Define an element $A_2\in \widetilde{S}_2(\thetagraph{4})_3$ by the formula 
 \[A_2=\sum_{\sigma\in \Sigma_{4}}\sgn(\sigma)(e_1-e_{\sigma(3)})(h_{1,\sigma(1)}-h_{1,1})(h_{2,\sigma(2)}-h_{2,1}).\]
\end{construction}

\begin{lemma}\label{lem:theta class}
The chain $A_2$ is a cycle, and $[A_2]$ generates $H_2(B_3(\thetagraph{4}))$. 
\end{lemma}
\begin{proof}
For the first claim, consideration of the transpositions $\tau_{13}$ and $\tau_{23}$ shows that all terms cancel in the sum \[\text{\small{$\partial A_2=\sum_{\sigma\in \Sigma_{4}}\sgn(\sigma)(e_1-e_{\sigma(3)})\left[(e_{\sigma(1)}-e_1) (h_{2,\sigma(2)}-h_{2,1})-(e_{\sigma(2)}-e_1)(h_{1,\sigma(1)}-h_{1,1})\right].$}}\] For the second claim, we privilege the half-edge $h_{2,1}$ and consider the long exact sequence arising from explosion of the vertex $v_2$, calculating that \begin{align*}
\psi([A_2])_{h_{2,4}}&=\left[\sum_{\sigma(2)=4}\sgn(\sigma)(e_1-e_{\sigma(3)})(h_{1,\sigma(1)}-h_{1,1})\right]\\
&=\left[\sum_{\sigma(2)=4,\,\sigma(4)=1}\sgn(\sigma)(e_1-e_{\sigma(3)})(h_{1,\sigma(1)}-h_{1,1})\right]\\
&=\alpha_{123}.
\end{align*} Similarly, we have $\psi([A_2])_{h_{2,2}}=\alpha_{134}$ and $\psi([A_2])_{h_{2,3}}=-\alpha_{124}$. It follows from the calculation made in the proof of Proposition \ref{prop:theta class} that $\psi([A_2])$ generates $\ker(\delta)$, implying the claim.
\end{proof}

\begin{remark}
One shows easily that $(e_i-e_j)[A_2]$ is a difference of two external products of star classes. This analogue of the $Q$-relation implies that, modulo tori, the action of $\mathbb{Z}[E]$ on a theta class factors through the quotient identifying $e_i$ for $1\leq i\leq 4$, implying that the graded dimension of the $\mathbb{Z}[E]$-submodule generated by a theta class grows at a strictly slower rate than $\dim H_2(B_k(\graf))$ (see \cite[Thm. 1.2]{AnDrummondColeKnudsen:ESHGBG}). Combined with Theorem \ref{thm:main}, this observation provides a hands-on verification of the degree 2 planar case of \cite[Thm. 1.1]{AnDrummondColeKnudsen:AHGBG}, which asserts that toric classes always dominate in the limit $k\to \infty$.
\end{remark}

\begin{remark}
The star and theta classes are the first two examples of a uniform construction outputting an element of $H_n(B_{n+1}(\completegraph{n,n+2}))$ for every $n>0$. We defer a systematic study of these classes to future work.
\end{remark}

\section{Decomposing graphs}\label{section:decomposing graphs}

In this section, we discuss two methods of reducing the complexity of a graph, which form the basis for various inductive arguments to come. The first technique involves the removal of a subset of vertices, while the second involves replacing a subgraph by an edge. The first is drawn from classical graph theory, and the second is discussed at greater leisure in \cite{AnDrummondColeKnudsen:ESHGBG}.

\subsection{Connectivity and cuts} Classical connectivity theory for graphs only behaves well after restricting to combinatorially well-behaved classes of graphs---simplicial graphs, for example. According to our conventions, not all graphs are simplicial complexes, since we allow self-loops and multiple edges, but a simple device will allow us to circumvent this difficulty.

\begin{definition}
A \emph{minimal simplicial model} for $\graf$ is a simplicial graph $\graf_\Delta$ homeomorphic to $\graf$ such that any smoothing with source $\graf_\Delta$ and simplicial target is an isomorphism.
\end{definition}

In view of the following standard result, we typically (and abusively) refer to $\graf_\Delta$ as \emph{the} minimal simplicial model.

\begin{proposition}\label{prop:simplicial model}
Every graph admits a minimal simplicial model, which is unique up to isomorphism.
\end{proposition}

Concretely, as long as $\graf$ has no component homeomorphic to $S^1$, the minimal simplicial model $\graf_\Delta$ may be obtained by the following three-step process: first, smooth all bivalent vertices of $\graf$; second, add a bivalent vertex to each self-loop of the resulting graph; third, add a bivalent vertex to all but one of each set of multiple edges of the resulting graph.

\begin{definition}
A non-singleton graph $\graf$ is (topologically) $k$-\emph{connected} if any two distinct vertices of $\graf_\Delta$ may be joined by $k$ paths pairwise disjoint away from the endpoints. By convention, the singleton graph is $1$-connected but not $k$-connected for any $k>1$.
\end{definition}

In the special cases $k=1,2,3$, we at times use the respective terms \emph{connected}, \emph{biconnected}, and \emph{triconnected}.

As illustrated by the classical theorem due to Menger \cite{Menger:SAK}, connectivity is intimately related to the concept of a cut in a graph.

\begin{definition}
A $k$-\emph{cut} in $\graf$ is a subset $S\subseteq V$ of cardinality $k$ such that the complement of the open star of $S$ has at least two connected components, each containing a vertex. If $S$ is a $k$-cut, an $S$-\emph{component} of $\graf$ is the closure in $\graf$ of a connected component of $\graf\setminus S$. 
\end{definition}

The set of $k$-cuts of a graph depends crucially on the combinatorial structure; for example, any tail may be subdivided to contain a bivalent $1$-cut. This pathology cannot occur in $\graf_\Delta$, so we define $N_1(\graf)$ to be one plus the number of $1$-cuts in $\graf_\Delta$, where $\graf$ is a connected graph. 

\begin{observation}
The parameter $N_1(\graf)$ has the following properties.
\begin{enumerate}
\item If $N_1(\graf)>1$, then $\graf_\Delta$ admits a $1$-cut $\{x\}$ such that $N_1(\graphfont{\Delta})<N_1(\graf)$ for every $\{x\}$-component $\graphfont{\Delta}$.
\item The equality $N_1(\graf)=1$ holds if and only if $\graf$ is biconnected, a singleton, or homeomorphic to $\completegraph{1,1}$.
\end{enumerate}
\end{observation}

As we shall see, the parameter $N_1(\graf)$ is useful in inductive arguments. We shall also make use of an analogue in the case $k=2$, for which we require an auxiliary definition.

\begin{definition}
Let $\{x,y\}$ be a $2$-cut in $\graf$ and $\graphfont{\Delta}$ an $\{x,y\}$-component. The \emph{completion} of $\graphfont{\Delta}$ is the graph $\overline{\graphfont{\Delta}}$ obtained from $\graphfont{\Delta}$ by adding an edge $e_{xy}$ joining $x$ and $y$.
\end{definition}

The following result is a consequence of the combinatorial decomposition theory of \cite{CunninghamEdmonds:CDT}, the details of which we elide.

\begin{theorem}[Cunningham--Edmonds]\label{thm:decomposition theory}
There is a parameter $N_2(\graf)$ associated to a biconnected graph $\graf$ with the following properties.
\begin{enumerate}
\item If $N_2(\graf)>1$, then $\graf_\Delta$ admits a $2$-cut $\{x,y\}$ such that $N_2(\overline{\graphfont{\Delta}})<N_2(\graf)$ for every $\{x,y\}$-component $\graphfont{\Delta}$.
\item The equality $N_2(\graf)=1$ holds if and only if $\graf$ is triconnected, a cycle, or homeomorphic to $\thetagraph{n}$ for some $n\geq 3$.
\end{enumerate}
\end{theorem}

\subsection{Surgery} In most circumstances, configuration spaces only enjoy functoriality for continuous injections. As we now explain, extra functoriality is available at the level of homology for configuration spaces of graphs.

\begin{definition}
Let $\graphfont{\Delta}\subseteq\graf$ be a connected subgraph with two distinct distinguished vertices $\{x,y\}$. The result of \emph{surgery on $\graf$ along $\graphfont{\Delta}$} (using the vertices $\{x,y\}$) is the graph obtained by replacing $\graphfont{\Delta}$ with a single edge $e_{xy}$ connecting $x$ and $y$. 
\end{definition}

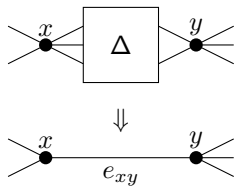
\begin{figure}[ht]
\begin{center}
\begin{tikzpicture}
\begin{scope}[xshift=6.5cm]
\begin{scope}[xshift=0cm]
\fill[black] (0,.5) circle (2.5pt) node[above] {$x$};
\draw(0,.5) -- (-.5,.25);
\draw(0,.5) -- (-.5,.75);
\draw(0,.5) -- (.5,.25);
\draw(0,.5) -- (.5,.5);
\draw(0,.5) -- (.5,.75);
\draw (.5,0) rectangle (1.5,1);
\fill[black] (2,.5) circle (2.5pt) node[above] {$y$};
\draw(2,.5) -- (1.5,.25);
\draw(2,.5) -- (1.5,.75);
\draw(2,.5) -- (2.5,.25);
\draw(2,.5) -- (2.5,.5);
\draw(2,.5) -- (2.5,.75);
\draw(1,.5)node{$\graphfont{\Delta}$};
\end{scope}
\draw(1,-.5)node{$\Downarrow$};
\begin{scope}[yshift=-1.5cm]
\fill[black] (0,.5) circle (2.5pt) node[above] {$x$};
\draw(0,.5) -- (-.5,.25);
\draw(0,.5) -- node[midway,below] {$e_{xy}$} (2,.5);
\draw(0,.5) -- (-.5,.75);
\fill[black] (2,.5) circle (2.5pt) node[above] {$y$};
\draw(2,.5) -- (2.5,.25);
\draw(2,.5) -- (2.5,.75);
\draw(2,.5) -- (2.5,.5);
\end{scope}
\end{scope}
\end{tikzpicture}
\end{center}
\caption{Depiction of a surgery}\label{figure: surgeries}
\end{figure}

Although a surgery is not a graph morphism in the sense given above, we nevertheless think of it as a type of morphism, writing $\sigma:\graf\dashrightarrow\graf'$ when $\graf'$ is the result of surgery on $\graf$. 

\begin{example}\label{example:completion surgery}
Given a $2$-cut $\{x,y\}$ in a biconnected graph $\graf$, there is a canonical sugery onto the completion of any $\{x,y\}$-component of $\graf$.
\end{example}

Our notation is justified by a certain non-obvious functoriality. Note that, given a path from $x$ to $y$ in $\graphfont{\Delta}$, there results a non-canonical topological embedding $\iota:\graf'\to \graf$, called a \emph{section} of the surgery $\sigma$.

\begin{proposition}\label{prop:surgery is a thing}
Let $\sigma:\graf\dashrightarrow\graf'$ be a surgery. There is a canonical map of bigraded Abelian groups $\sigma_*:H_*(B(\graf))\to H_*(B(\graf'))$ that is a retraction of $\iota_*$ for any section $\iota:\graf'\to \graf$. Moreover, $\sigma_*$ is compatible with edge stabilization via the ring homomorphism \[\sigma_*(e)=\begin{cases}
e&\quad e\notin E(\graphfont{\Delta})\\
e_{xy} &\quad e\in E(\graphfont{\Delta}).
\end{cases}\]
\end{proposition}

This result is a consequence of the proof of \cite[Lem. 4.17]{AnDrummondColeKnudsen:ESHGBG}, where it is shown that $\sigma_*$ is induced by a map $\intrinsic{\sigma}$ at the level of \'{S}wi\k{a}tkowski complexes. Concretely, the map $\intrinsic{\sigma}$ sends edges of $\graphfont{\Delta}$ and vertices of $\graphfont{\Delta}$ different from $x$ and $y$ to $e_{xy}$; sends half-edges of $\graphfont{\Delta}$ incident on $x$ to the unique half-edge of $e_{xy}$ incident on $x$ (resp. $y$); annihilates all other half-edges of $\graphfont{\Delta}$; and acts as the identity on vertices, edges, and half-edges not lying in $\Delta$.

We close with two results concerning the difference between a homology class $\alpha$ and its modification $\iota_*\sigma_*(\alpha)$ after surgery. In both statements, we begin with a decomposition $\graf=\graphfont{\Delta}\cup \graphfont{\Delta}'$, where $\graphfont{\Delta}$ and $\graphfont{\Delta}'$ are subgraphs with $\graphfont{\Delta}$ connected.

\begin{proposition}\label{prop:very basic surgery}
Suppose that $\graphfont{\Delta}$ and $\graphfont{\Delta'}$ intersect in the single vertex $x$, choose a vertex $x\neq y\in \graphfont{\Delta}$ arbitrarily, and let $\sigma:\graf\dashrightarrow \graf'$ denote the resulting surgery along $\graphfont{\Delta}$. Fix a star cycle $a$, a loop cycle $b$, and $p,q\in\mathbb{Z}[E]$. If the support of $a$ lies $\graphfont{\Delta}'\setminus\{x\}$, then $(1-\iota_*\sigma_*)([pa])=0$ for any section $\iota:\graf'\to\graf$ of $\sigma$ (resp. $b$, $[qb]$).
\end{proposition}
\begin{proof}
Since $x$ does not lie in the support of $a$, we may assume by connectivity of $\graphfont{\Delta}$ that $p$ involves only edges lying in $\graphfont{\Delta}'$, in which case the claim holds by inspection (resp. $b$, $q$).
\end{proof}

\begin{proposition}\label{prop:basic surgery}
Suppose that $\graphfont{\Delta}$ and $\graphfont{\Delta'}$ intersect in the pair of distinct vertices $\{x,y\}$, and let $\sigma:\graf\dashrightarrow \graf'$ denote the resulting surgery along $\graphfont{\Delta}$. Fix a loop cycle $b$ in $\graf$ and $q\in\mathbb{Z}[E]$, and suppose that the support of $b$ does not lie entirely in $\graphfont{\Delta}$. There is a section $\iota:\graf'\to \graf$ of $\sigma$ such that $(1-\iota_*\sigma_*)([qb])$ is a sum of stabilized star classes represented by cycles with support lying in $\graphfont{\Delta}$.
\end{proposition}
\begin{proof}
By assumption, the support of $b$ intersects $\graphfont{\Delta}$ either in a subset of $\{x,y\}$ or in a path from $x$ to $y$. Choose $\iota$ such that this path coincides with $\iota(e_{xy})$ in the latter case and arbitrarily otherwise. 
By inspection, $[qb]=\iota_*\sigma_*([qb])$ in the special case where $q$ involves only edges lying either in $\graphfont{\Delta}'$ or in the support of $b$. By connectivity and the Q-relation, $q$ may be assumed of this form at the cost of introducing star classes with the desired support property.
\end{proof}

\section{Pesky cycles}\label{section:pesky cycles}
In this section, we reformulate Theorem \ref{thm:main} in terms of the surjectivity of a certain map. We then prove Proposition \ref{prop:cokernel generators}, which describes generators for the cokernel of the map in question. These ``pesky'' cycles are the main players in the remainder of the paper, which is devoted to the proof of their vanishing.

\subsection{Reformulation} Given a graph $\graf$, we define $M(\graf)\subseteq H_2(B(\graf))$ to be the $\mathbb{Z}[E]$-submodule generated by theta classes and external products of loop classes and star classes; thus, the conclusion of Theorem \ref{thm:main} is that $M(\graf)=H_2(B(\graf))$. Given a bivalent vertex $w$, we consider the long exact sequence
\[\cdots \to H_2(B_k(\graf_w))\xrightarrow{\iota_*}H_2(B_k(\graf))\xrightarrow{\psi} H_1(B_{k-1}(\graf_w))\xrightarrow{\delta} H_1(B_{k}(\graf_w))\to \cdots\] of Proposition \ref{prop:vertex explosion}. We will prove the following result concerning this sequence.

\begin{theorem}\label{thm:reformulation}
For any connected planar graph $\graf$ and bivalent vertex $w$, the restriction $\psi:M(\graf)\to \ker(\delta)$ is surjective.
\end{theorem}

We now clarify the relationship between this result and the main theorem, beginning with the following simple observation.

\begin{lemma}\label{lem:connected reduction}
Let $\graf_1$ and $\graf_2$ be planar graphs. The conclusion of Theorem \ref{thm:main} holds for $\graf_1\sqcup \graf_2$ if and only if it holds for $\graf_1$ and $\graf_2$.
\end{lemma}
\begin{proof}
The claim follows from the homeomorphism $B\left(\graf_1\sqcup\graf_2\right)\cong B(\graf_1)\times B(\graf_2)$ and the K\"{u}nneth isomorphism. The latter holds integrally since $H_1(B(\graf_i))$ is torsion-free for $i\in\{1,2\}$ by planarity \cite[Cor. 3.6]{KoPark:CGBG}.
\end{proof}

\begin{proposition}\label{prop:implication}
If $\graf$ is a graph with a bivalent vertex $w$ such that the conclusion of Theorem \ref{thm:main} holds for $\graf_w$, then the conclusion of Theorem \ref{thm:main} also holds for $\graf$ provided either
\begin{enumerate}
\item $\graf_w$ is connected and the conclusion of Theorem \ref{thm:reformulation} holds for $\graf$ and $w$, or
\item $\graf_w$ is disconnected.
\end{enumerate}
\end{proposition}
\begin{proof}
In the first case, we have $\iota_*(H_2(B(\graf_w)))=\iota_*(M(\graf_w))\subseteq M(\graf)$. From the long exact sequence above, the map $\psi$ induces an isomorphism \[\frac{H_2(B(\graf))}{\iota_*(H_2(B(\graf_w)))}\cong \ker(\delta)\implies \frac{H_2(B(\graf))}{M(\graf)}\cong \frac{\ker(\delta)}{\psi(M(\graf))}=0,\] whence $H_2(B(\graf))=M(\graf)$, as desired.

In the second case, Lemma \ref{lem:connected reduction} implies that the conclusion of Theorem \ref{thm:main} holds for each component of $\graf_w$, and the claim follows from \cite[Prop. 5.22]{AnDrummondColeKnudsen:SSGBG}. 
\end{proof}

\begin{corollary}\label{cor:equivalent}
Theorems \ref{thm:main} and \ref{thm:reformulation} are equivalent.
\end{corollary}
\begin{proof}
The forward implication is immediate from exactness. For the reverse implication, we proceed by induction on the first Betti number $\beta_1$ of $\graf$. The base case of $\beta_1=0$ is that of a tree, which is well known (see also Example \ref{example:tree}). For the induction step, choose a bivalent vertex $w$ such that $\graf$ is connected, subdividing if necessary. Since the conclusion of Theorem \ref{thm:main} holds for $\graf_w$ by induction, Proposition \ref{prop:implication} yields the conclusion.
\end{proof}

In pursuing Theorem \ref{thm:reformulation}, we will use the following criterion repeatedly. Here we abusively write $w$ and $w'$ for the vertices of $\graf_w$ corresponding to the two half-edges of $\graf$ incident on $w$.

\begin{lemma}[Path argument]\label{lem:path argument}
Let $\gamma$ be a path from $w$ to $w'$ in $\graf_w$. A loop or star class lies in $\psi(M(\graf))$ if it admits a representing cycle with support disjoint from $\gamma$.
\end{lemma}
\begin{proof}
Let $a$ denote the representing cycle. From the path $\gamma$, we obtain a loop in $\graf$ and thereby a loop cycle. By the assumption on the support of $a$, the external product of $a$ with this loop cycle is defined, and the homology class of the external product lies in $M(\graf)\cap \psi^{-1}([a])$ by inspection of the definition of $\psi$.
\end{proof}

As an immediate application, we obtain the following special case of Theorem \ref{thm:reformulation}.

\begin{corollary}\label{cor:vertices coincide}
Let $\graf$ be a connected graph with a bivalent vertex $w$. If the edges incident on $w$ form a loop, then the conclusion of Theorem \ref{thm:reformulation} holds for $\graf$ and $w$.
\end{corollary}
\begin{proof}
Let $e$ and $e'$ denote the edges incident on $w$. By assumption, there is a vertex $u\neq w$ such that $e$ and $e'$ are also incident on $u$. Then $e$ and $e'$ lie in distinct components of $(\graf_w)_u$, so Lemma \ref{lem:1-cut torsion} implies that $\ker(\delta)$ lies in the image of $H_1(B((\graf_w)_u))$, and the path argument applies.
\end{proof}

We close with an analogue of the path argument for star classes.

\begin{lemma}\label{lem:non-rigid in image}
If $a$ is a non-rigid star cycle, then $[a]\in \psi(M(\graf))$.
\end{lemma}
\begin{proof}
Let $u$ denote the support of $a$. By non-rigidity, there is an embedding of a subdivision of $\thetagraph{3}$ into $\graf_w$ under which $u$ is the image of an essential vertex. Denote the image of the other essential vertex by $u'$, and choose a path in $\graf_w$ from $w$ to $w'$. If the path avoids $u'$, then the $\theta$-relation and the path argument yield the conclusion. If the path contains $u'$, then our embedding extends to an embedding of a subdivision of $\thetagraph{4}$ into $\graf$, and Lemma \ref{lem:theta les} implies the claim.
\end{proof}

\subsection{Standard and pesky cycles}\label{section:cycles}

The goal of this section is to describe generators for the cokernel of the map of Theorem \ref{thm:reformulation}. Before giving a precise formulation in Proposition \ref{prop:cokernel generators} below, we require a few preliminary definitions and results. We begin by establishing notation that we maintain for the remainder of the paper.

\begin{notation}[see Figure \ref{fig:setup}]\label{notation:setup} Let $\graf$ be a connected graph. Fix an edge $e$ with vertices $v$ and $v'$, and subdivide $e$ by adding a bivalent vertex $w$. Abusively, we write $w$ and $w'$ for the resulting vertices of $\graf_w$ and $e$ and $e'$ for the resulting edges, so that $e$ has vertices $v$ and $w$ (resp. $e'$, $v'$ and $w'$).
We write $\graf_e\coloneqq\graf\setminus e$ for the graph given by the complement of the (open) edge $e$.
\end{notation}

As long as $\graf$ has an essential vertex, we may assume after smoothing that $v$ and $v'$ are essential in $\graf$ and in $\graf_w$; however, either vertex may be bivalent in $\graf_e$. In light of Corollary \ref{cor:vertices coincide}, we further assume that $v\neq v'$, i.e., that $e$ is not a self-loop in $\graf$.

\begin{figure}
\begin{align*}
\graf&=
\begin{tikzpicture}[baseline=-.5ex]
\draw(-0.5,-0.5) rectangle (0.5,0.5);
\fill (-0.5,0) circle (2pt) node[above left] {$v$} (0.5,0) circle (2pt) node[above right] {$v'$};
\draw[rounded corners] (-0.5,0) -- (-1,0) -- (-1,-1) -- node[midway, below] {$e$} (1,-1) -- (1,0) -- (0.5,0);
\end{tikzpicture}&
\graf_w&=
\begin{tikzpicture}[baseline=-.5ex]
\draw(-0.5,-0.5) rectangle (0.5,0.5);
\fill (-0.5,0) circle (2pt) node[above left] {$v$} (0.5,0) circle (2pt) node[above right] {$v'$};
\draw[rounded corners] (-0.5,0) -- (-1,0) -- node[midway, left] {$e$} (-1,-1) (0.5,0) -- (1,0) -- node[midway, right] {$e'$} (1,-1);
\fill (-1,-1) circle (2pt) node[left] {$w$} (1,-1) circle (2pt) node[right] {$w'$};
\end{tikzpicture}&
\graf_e&=
\begin{tikzpicture}[baseline=-.5ex]
\draw(-0.5,-0.5) rectangle (0.5,0.5);
\fill (-0.5,0) circle (2pt) node[above left] {$v$} (0.5,0) circle (2pt) node[above right] {$v'$};
\end{tikzpicture}
\end{align*}
\caption{Schematic depiction of Notation \ref{notation:setup}}\label{fig:setup}
\end{figure}
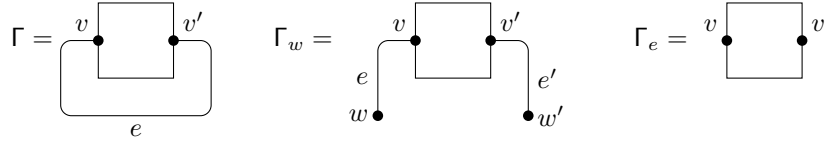

\begin{definition} Let $\graf_w$ be as in Notation \ref{notation:setup}.
\begin{enumerate}
\item A \emph{standard cycle} is a weight-homogeneous $1$-cycle $c\in S_1(\graf_w)$ of the form \[c=\sum_{i=1}^rp_ia_i+\sum_{j=1}^sq_jb_j,\] where the $a_i$ are star cycles, the $b_j$ loop cycles, and the $p_i$ and $q_j$ polynomials in the edges of $\graf_w$. The \emph{support} of $c$ is the union of the supports of the $a_i$ and $b_j$.
\item We define the following potential properties of a standard cycle $c$.
\begin{itemize}
\item[{\bf(P)}] Every {\bf path} from $w$ to $w'$ in $\graf_w$ intersects the support of every $a_i$ and every $b_j$.
\item[{\bf(E)}] If an {\bf edge} $e$ is involved in $p_i$ (resp. $q_j$), then $e$ is adjacent to (resp. contained in) the support of $a_i$ (resp. $b_j$).
\item[{\bf(S)}] No $1$-cut of $\graf_w$ separating $e$ and $e'$ is contained in the {\bf support} of $c$.
\item[{\bf(K)}] The class $[c]$ lies in the {\bf kernel} of $\delta$.
\end{itemize}
\item If $c$ has all four of these properties, then we say that $c$ is \emph{pesky}.
\end{enumerate}
\end{definition}

The interest of pesky cycles lies in the following result.

\begin{proposition}\label{prop:cokernel generators}
The quotient $\ker(\delta)/\psi(M(\graf))$ is generated by the images of pesky cycles with no star summands.
\end{proposition}

The proof requires a few simple lemmas.

\begin{lemma}\label{lem:cuts}
If $c$ is a standard cycle satisfying {\bf K}, then $c$ is homologous to a standard cycle satisfying {\bf S} and {\bf E}.
\end{lemma}
\begin{proof}
Let $W$ be the set of $1$-cuts of $\graf_w$ separating $e$ and $e'$. Repeated use of Lemma \ref{lem:1-cut torsion} shows that $[c]$ lies in the image of $H_1(B((\graf_w)_W))$. By Proposition \ref{prop:stars and loops generate}, we may represent $[c]$ by a standard cycle with support avoiding $W$, and the $Q$-relation and connectedness of $\graf_w$ imply that we may take this standard cycle to satisfy {\bf E}.
\end{proof}

\begin{lemma}\label{lem:support}
Let $c$ be a standard cycle. Then $c=c_1+c_2$, where $[c_1]\in \psi(M(\graf))$ and $c_2$ is a standard cycle satisfying {\bf P}. Moreover, if $c$ satisfies {\bf E}, {\bf S}, or {\bf K}, then so does $c_2$.
\end{lemma}
\begin{proof}
Let $I$ denote the set of indices $i$ for which there is a path from $w$ to $w'$ in $\graf_w$ avoiding the support of $a_i$ (resp. $J$, $j$, $b_j$), and define \begin{align*}c_1&=\sum_{i\in I} p_ia_i+\sum_{j\in J} q_jb_j\\
c_2&=c-c_1.
\end{align*} The path argument implies that $[c_1]\in \psi(M(\graf))$, and, by construction, $c_2$ is a standard cycle satisfying {\bf P} and further satisfying {\bf E} or {\bf S} if $c$ did. For the last claim, we note that $\delta\circ \psi=0$, so $\delta([c])=\delta([c_2])$.
\end{proof}

\begin{lemma}\label{lem:pesky star}
A pesky cycle contains no star summands.
\end{lemma}
\begin{proof}
Supposing otherwise, it follows from {\bf P} that the support of the star cycle in question is a $1$-cut separating $e$ and $e'$, violating {\bf S}.
\end{proof}

\begin{proof}[Proof of Proposition \ref{prop:cokernel generators}]
The claim is trivial unless $\graf$ has an essential vertex, in which case Lemmas \ref{lem:cuts}, \ref{lem:support}, and \ref{lem:pesky star} imply the claim. 
\end{proof}

\section{Examples and a first reduction}\label{section:examples and a first reduction}

The proof of Theorem \ref{thm:reformulation} proceeds through two inductions. In this section, we carry out the first (and simplest) of these inductions after considering various examples related to base cases.

\subsection{Rogues' gallery}\label{section:rogues gallery} We pause to consider a few (partially redundant) examples that will be of use in what follows.


\begin{example}\label{example:tree}
If $\graf$ is a tree, then the conclusion of Theorem \ref{thm:main} holds for $\graf$. This well known claim follows (for example) by repeated application of \cite[Prop. 5.22]{AnDrummondColeKnudsen:SSGBG}.
\end{example}

\begin{example}\label{example:tree with loops}
If $\graf$ is obtained from a tree by attaching self-loops, then the conclusion of Theorem \ref{thm:main} holds for $\graf$. This claim follows from Proposition \ref{prop:implication} by induction on the number of self-loops using Example \ref{example:tree} as the base case and Corollary \ref{cor:vertices coincide} for the induction step.
\end{example}

\begin{example}\label{example:one loop}
If the first Betti number $b_1$ of $\graf$ is equal to $1$, then the conclusion of Theorem \ref{thm:main} holds for $\graf$. This claim follows from Example \ref{example:tree} and Proposition \ref{prop:cokernel generators} after exploding a bivalent vertex $w$ such that $\graf_w$ is a tree. Indeed, a tree admits no loop cycles and no star cycles satisfying {\bf S}, hence no nonzero pesky cycles.
\end{example}

\begin{example}\label{example:unitrivalent}
If $\graf$ has exactly two essential vertices, each trivalent, then $H_2(B(\graf))$ is generated by external products of star classes and loop classes by \cite[Prop. 5.25]{AnDrummondColeKnudsen:SSGBG}, so the conclusion of Theorem \ref{thm:main} holds in this case.
\end{example}

\begin{example}\label{example:theta}
More generally, the conclusion of Theorem \ref{thm:main} holds if $\graf$ has exactly two essential vertices. Assume first that $\graf$ has no self-loops (see Figure \ref{fig:theta with markings}). Adding a bivalent vertex $w$ (shown in green) to one of the non-tail edges, we observe that the support of every loop and star cycle in $\graf_w$ contains one of the essential vertices of $\graf_w$, thereby violating {\bf S}. Thus, $\graf_w$ has no nonzero pesky cycles, and it follows from Proposition \ref{prop:cokernel generators} that the conclusion of Theorem \ref{thm:reformulation} holds. The claim in the case without self-loops now follows from Proposition \ref{prop:implication} by induction on the first Betti number using Example \ref{example:tree} as a base case. The claim in general follows by repeated application of Corollary \ref{cor:vertices coincide} as in Example \ref{example:tree with loops}.
\end{example}

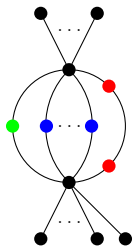
\begin{figure}[ht]
\begin{tikzpicture}[scale=1.5]
\fill[black] (-.25,1) circle (5/3 pt);
\fill[black] (.25,1) circle (5/3 pt);
\fill[black] (-.25,-1) circle (5/3 pt);
\fill[black] (.25,-1) circle (5/3 pt);
\fill[black] (.5,-1) circle (5/3 pt);
\fill[black] (0,-.5) circle (5/3pt);
\fill[black] (0,.5) circle (5/3pt);
\draw(0,0) circle (.5cm);
\draw(0,.5) to [bend right=45] (0,-.5);
\node at (0,0) {\,{\tiny$\cdots$}};
\node at (0,.85) {\,{\tiny$\cdots$}};
\node at (0,-.85) {\,{\tiny$\cdots$}};
\draw(0,.5) to [bend left=45] (0,-.5);
\draw(0,.5) to (-.25,1);
\draw(0,.5) to (.25,1);
\draw(0,-.5) to (-.25,-1);
\draw(0,-.5) to (.25,-1);
\draw(0,-.5) to (.5,-1);
\fill[green] (-.5,.0) circle (5/3pt);
\fill[blue] (-.2,.0) circle (5/3pt);
\fill[red] (.354,.354) circle (5/3pt);
\fill[red] (.354,-.354) circle (5/3pt);
\fill[blue] (.2,.0) circle (5/3pt);
\end{tikzpicture}
\caption{The graph of Example \ref{example:theta} with auxiliary markings}
\label{fig:theta with markings}
\end{figure}

\begin{example}\label{example:weird theta 1}
The conclusion of Theorem \ref{thm:main} holds for the graph obtained by adding an edge $e_0$ joining the blue vertices in the graph of Figure \ref{fig:theta with markings}. The argument here is essentially the same; we induct on the first Betti number, reducing to the case $b_1=2$ by observing that any loop or star cycle in $\graf_w$ violates either {\bf P} or {\bf S}. In the case $b_1=2$, adding the vertex $w$ to $e_0$ again yields an explosion $\graf_w$ with no nonzero pesky cycles, and, since the conclusion holds for $\graf_w$ by Example \ref{example:one loop}, the claim follows.
\end{example}

\begin{example}\label{example:weird theta 2}
The conclusion of Theorem \ref{thm:main} holds for the graph obtained by adding an edge $e_0$ joining a red vertex in the graph of Figure \ref{fig:theta with markings} either to an essential vertex or to the other red vertex. As before, the claim follows by induction on $b_1$ via a sequence of explosions such that $\graf_w$ has no nonzero pesky cycles. 
\end{example}

We close with an example showing that the assumption of planarity in Theorem \ref{thm:main} is necessary.

\begin{example}\label{example:non-planar}
Let $\graphfont{\Delta}$ be the union of a complete bipartite graph $\completegraph{3,3}$ and a star graph $\stargraph{3}$ along a set of three pairwise non-adjacent vertices---see Figure~\ref{figure:doubled k33}. Subdividing and exploding one of the edges $\stargraph{3}$, we make two claims: first, the loop cycle $b$ shown in red is pesky; second, $M(\graphfont{\Delta})$ vanishes in weight $2$. In light of Proposition \ref{prop:cokernel generators}, these claims together imply that the conclusion of Theorem \ref{thm:reformulation}, hence that of Theorem \ref{thm:main}, fails for $\graphfont{\Delta}$.

\begin{figure}[ht]
\centering
\[
\graphfont{\Delta}=\begin{tikzpicture}[baseline=-0.5ex,scale=1]
\fill(0,0) circle (2pt);
\draw(0,0) -- node[midway,above left] {$e$} (1,1) (0,0) -- (1,0) (0,0) -- (1,-1);
\foreach \i in {-1,...,1} {
	\fill (1,\i) circle (2pt);
	\fill (2,\i) circle (2pt);
	\foreach \j in {-1,...,1} {
		\draw(1,\i) -- (2,\j);
	}
}
\draw[thick,red,->] (1.5,-1) -- (2,-1) -- (1,0) -- (2,0) -- (1,-1) -- (1.5,-1) node[below] {$b$};
\end{tikzpicture}\qquad
\graphfont{\Delta}_w=\begin{tikzpicture}[baseline=-0.5ex,scale=1]
\fill(0,0) circle (2pt);
\fill(0.3,1) circle (2pt) node[above left] {$w$};
\fill(0,0.7) circle (2pt) node[above left] {$w'$};
\draw(0,0) -- node[midway,left] {$e'$} (0,0.7);
\draw(0.3,1) -- node[midway,above] {$e$} (1,1) (0,0) -- (1,0) (0,0) -- (1,-1);
\foreach \i in {-1,...,1} {
	\fill (1,\i) circle (2pt);
	\fill (2,\i) circle (2pt);
	\foreach \j in {-1,...,1} {
		\draw(1,\i) -- (2,\j);
	}
}
\draw[thick,red,->] (1.5,-1) -- (2,-1) -- (1,0) -- (2,0) -- (1,-1) -- (1.5,-1) node[below] {$b$};
\end{tikzpicture}
\]
\caption{The graph $\graphfont{\Delta}$ and a loop cycle in red}
\label{figure:doubled k33}
\end{figure}
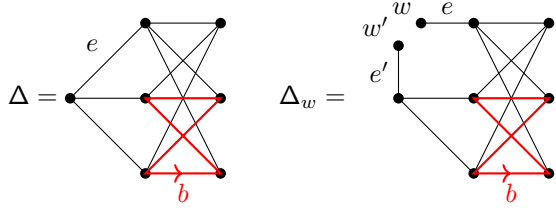

For the first claim, only {\bf K} is not immediate. Writing $\beta=[b]$, the Q-relation implies that
\begin{align*}
(e-e')\beta&=(e-e_0)\beta-(e'-e_0)\beta=\alpha-\alpha',
\end{align*}
where $e_0$ is a fixed edge in the support of $b$, and $\alpha$ and $\alpha'$ are star classes. Since $\alpha=\alpha'$ by Example \ref{example:K33 stars}, the claim follows.

For the second claim, $M(\graphfont{\Delta})$ is spanned in weight $2$ by external products of loop classes, since theta classes or external products involving a star class require at least $3$ particles (this much is true for any graph). On the other hand, there is no pair of disjoint loop cycles in $\graphfont{\Delta}$, since every loop in $\graphfont{\Delta}$ involves at least $4$ of the $7$ vertices.
\end{example}

\begin{remark}
The calculation of Example \ref{example:non-planar} exhibits a previously unknown type of class.
\end{remark}

\subsection{Reduction to the biconnected case} We close this section with a reduction permitting the full use of Proposition \ref{prop:cokernel generators}. The reader is reminded that we maintain Notation \ref{notation:setup}.

\begin{lemma}\label{lem:reduce to biconnected}
If Theorem \ref{thm:reformulation} holds under the further assumption that $\graf_e$ is biconnected, then it holds in general.
\end{lemma}
\begin{proof}
We proceed by induction on $N_1(\graf_e)$. In the base case of $N_1(\graf_e)=1$, either $\graf_e$ is biconnected, in which case our assumption applies; or $\graf_e$ is an isolated vertex, in which case $\graf$ is homeomorphic to a cycle graph; or $\graf_e$ is homeomorphic to an interval. In this latter case, either $\graf$ is homeomorphic to a cycle graph, or $\graf$ has two essential vertices, each trivalent. The examples of Section \ref{section:rogues gallery} encompass these cases.

Given a $1$-cut $u$ of $\graf_e$ (which is necessarily a $1$-cut of $\graf_w$, and which may coincide with $v$ or $v'$), \cite[Lem. 3.11]{KoPark:CGBG} supplies the decomposition \[H_1(B_k(\graf_w))\cong \left(\bigoplus_{\ell=1}^m H_1(B_k(\graphfont{\Delta}_{\ell}))\oplus \mathbb{Z}^{N(k,\graf_w, u)}\right)\] of Abelian groups, where the $\graphfont{\Delta}_{\ell}$ are the $u$-components of $\graf_w$, and the last term is spanned by star classes at $u$ (see also \cite[Lem. 3.11]{AnDrummondColeKnudsen:SSGBG}). We consider the relationship between the terms of this direct sum decomposition and those of a pesky cycle $c=\sum_{i=1}^rp_ia_i+\sum_{j=1}^sq_jb_j$.

There are two cases. If $u$ separates $e$ and $e'$, then {\bf S} implies that $u$ does not lie in the support of $c$. Thus, each term lies in one of the first $m$ summands. Since $H_1(B_k(\graphfont{\Delta}_{\ell}))\subseteq \psi(M(\graf))$ by the path argument and Proposition \ref{prop:stars and loops generate} if $\graphfont{\Delta}_{\ell}$ does not contain $e$ or $e'$, we may assume by naturality of $\psi$ that $\ell=2$. By Proposition \ref{prop:very basic surgery}, we may write $[c]=\iota^{1}_*\sigma^{1}_*([c])+\iota^{2}_*\sigma^{2}_*([c])$, where $\sigma^{\ell}$ is surgery along $\graphfont{\Delta}_{\ell}$ using the vertices $u$ and either $w$ or $w'$, as appropriate. The claim now follows from $\mathbb{Z}[E]$-linearity of $\sigma^{\ell}_*$, the induction hypothesis, and naturality of $\psi$.

If $u$ does not separate $e$ and $e'$, we may assume by the path argument and naturality of $\psi$ that $\ell=1$, and the claim follows by induction.
\end{proof}

The remainder of the paper is devoted to establishing Theorem \ref{thm:reformulation} under the assumption that $\graf_e$ is biconnected. Although necessarily more elaborate in its details, the strategy is essentially the same, namely to reduce to the triconnected case by induction on the parameter $N_2(\graf_e)$. The primary advantage afforded by biconnectivity is that, through {\bf E} and Proposition \ref{prop:cokernel generators}, each summand of a pesky cycle may be taken to be a loop cycle stabilized at edges internal to its own support.

\section{Base cases}\label{section:base cases}

The goal of this section is to prove Theorem \ref{thm:reformulation} under the assumption that $\graf_e$ is triconnected, a cycle graph, or a theta graph (we maintain Notation \ref{notation:setup}). These cases will form the basis for the induction carried out in Section \ref{section:induction}.

\subsection{The triconnected case} In this section, we prove the following result.

\begin{lemma}\label{lem:triconnected case}
Suppose that $\graf$ is planar and $\graf_e$ is triconnected. If $b$ is a loop cycle in $\graf_w$ with support disjoint from $\{v,v'\}$, then $[b]\in \psi(M(\graf))$.
\end{lemma}

Before turning to the proof, we establish notation. Using the connectivity assumption, we may find paths $\gamma_1$ and $\gamma_2$ from $v$ to $v'$ in $\graf_e$ disjoint away from $\{v,v'\}$. This claim follows after applying the definition of triconnectivity to the nearest essential vertices to $v$ and $v'$ (which may or may not be $v$ and $v'$ themselves).

By the path argument, we may assume that $b$ satisfies {\bf P}, so the support of $b$ intersects both paths, and we write $x_i$ and $y_i$ for the first and last vertices of the intersection with $\gamma_i$. Although no two points with different index coincide, it may be that $x_i=y_i$. Thus, $\graf$ contains a topologically embedded \emph{angel graph}---see Figure \ref{fig:angels}.

\begin{figure}[ht]
\begin{align*}
&\begin{tikzpicture}[baseline=-.5ex]
\draw[thick] (0,0) circle (0.5);
\draw[fill,thick] (-1,-1) circle (2pt) node[left] {$v$} -- (1,-1) circle (2pt) node[right] {$v'$};
\draw (0,-1) node[below] {$e$};
\draw[thick] (-1,-1) to[out=90,in=135] (135:0.5);
\draw[fill] (135:0.5) circle (2pt) node[above] {$x_1$};
\draw[thick] (1,-1) to[out=90,in=45] (45:0.5);
\draw[fill] (45:0.5) circle (2pt) node[above] {$y_1$};
\draw[thick] (-1,-1) to[out=45,in=-135] (-135:0.5);
\draw[fill] (-135:0.5) circle (2pt) node[below] {\,\,$x_2$};
\draw[thick] (1,-1) to[out=135,in=-45] (-45:0.5);
\draw[fill] (-45:0.5) circle (2pt) node[below] {$y_2$};
\end{tikzpicture}&
&\begin{tikzpicture}[baseline=-.5ex]
\draw[thick] (0,0) circle (0.5);
\draw[fill,thick] (-1,-1) circle (2pt) node[left] {$v$} -- (1,-1) circle (2pt) node[right] {$v'$};
\draw (0,-1) node[below] {$e$};
\draw[thick] (-1,-1) to[out=90,in=135] (135:0.5);
\draw[fill] (135:0.5) circle (2pt) node[above] {$x_1$};
\draw[thick] (1,-1) to[out=90,in=45] (45:0.5);
\draw[fill] (45:0.5) circle (2pt) node[above] {$y_1$};
\draw[thick] (-1,-1) -- (0,-0.5);
\draw[thick] (1,-1) -- (0,-0.5);
\draw[fill] (0,-0.5) circle (2pt) node[above] {$x_2$};
\end{tikzpicture}&
&\begin{tikzpicture}[baseline=-.5ex]
\draw[thick] (0,0) circle (0.5);
\draw[fill,thick] (-1,-1) circle (2pt) node[left] {$v$} -- (1,-1) circle (2pt) node[right] {$v'$};
\draw (0,-1) node[below] {$e$};
\draw[thick] (-1,-1) -- (-1,0.5) arc (180:0:0.5);
\draw[thick] (1,-1) -- (1,0.5) arc (0:180:0.5);
\draw[fill] (0,0.5) circle (2pt) node[below] {$x_1$};
\draw[thick] (-1,-1) -- (0,-0.5);
\draw[thick] (1,-1) -- (0,-0.5);
\draw[fill] (0,-0.5) circle (2pt) node[above] {$x_2$};
\end{tikzpicture}
\end{align*}
\caption{Angels}
\label{fig:angels}
\end{figure}

Note that the four points carry a natural cyclic ordering up to reversal of orientation, which must be $(x_1, y_1, y_2, x_2)$ by planarity; in particular, there is no case in which the pairs of endpoints are linked.

The core of the argument is an analysis of such embedded graphs in the presence of sufficient connectivity. For the sake of narrative flow, this analysis is deferred to Section \ref{section:angelic graphs} below. There we prove Proposition \ref{prop:angel}, which implies that our embedded angel graph extends to an embedding of one of the graphs depicted in Figure \ref{fig:split angels} (we have shown only the ``generic'' case---in the case where $x_1=y_1$ and $x_2=y_2$, only the first possibility occurs, and so forth).

\begin{figure}[ht]
\begin{align*}
&\begin{tikzpicture}[baseline=-.5ex]
\draw[thick] (0,0) circle (0.5);
\draw[fill,thick] (-1,-1) circle (2pt) -- (1,-1) circle (2pt);
\draw[thick] (-1,-1) to[out=90,in=135] (135:0.5);
\draw[fill] (135:0.5) circle (2pt);
\draw[thick] (1,-1) to[out=90,in=45] (45:0.5);
\draw[fill] (45:0.5) circle (2pt);
\draw[thick] (-1,-1) to[out=45,in=-135] (-135:0.5);
\draw[fill] (-135:0.5) circle (2pt);
\draw[thick] (1,-1) to[out=135,in=-45] (-45:0.5);
\draw[fill] (-45:0.5) circle (2pt);
\draw[thick,red] (-0.5,0) -- (0.5,0);
\draw[fill,red] (0:0.5) circle (2pt) (180:0.5) circle (2pt);
\end{tikzpicture}&
&\begin{tikzpicture}[baseline=-.5ex]
\draw[thick] (0,0) circle (0.5);
\draw[fill,thick] (-1,-1) circle (2pt) -- (1,-1) circle (2pt);
\draw[thick] (-1,-1) to[out=90,in=135] (135:0.5);
\draw[fill] (135:0.5) circle (2pt);
\draw[thick] (1,-1) to[out=90,in=45] (45:0.5);
\draw[fill] (45:0.5) circle (2pt);
\draw[thick] (-1,-1) to[out=45,in=-135] (-135:0.5);
\draw[fill] (-135:0.5) circle (2pt);
\draw[thick] (1,-1) to[out=135,in=-45] (-45:0.5);
\draw[fill] (-45:0.5) circle (2pt);
\draw[thick,red] (0, 0.5) arc (180:270:0.5);
\draw[fill,red] (0,0.5) circle (2pt);
\draw[fill,red] (0.5,0) circle (2pt);
\end{tikzpicture}&
&\begin{tikzpicture}[baseline=-.5ex]
\draw[thick] (0,0) circle (0.5);
\draw[fill,thick] (-1,-1) circle (2pt) -- (1,-1) circle (2pt);
\draw[thick] (-1,-1) to[out=90,in=135] (135:0.5);
\draw[fill] (135:0.5) circle (2pt);
\draw[thick] (1,-1) to[out=90,in=45] (45:0.5);
\draw[fill] (45:0.5) circle (2pt);
\draw[thick] (-1,-1) to[out=45,in=-135] (-135:0.5);
\draw[fill] (-135:0.5) circle (2pt);
\draw[thick] (1,-1) to[out=135,in=-45] (-45:0.5);
\draw[fill] (-45:0.5) circle (2pt);
\draw[thick,red] (0, -0.5) arc (0:90:0.5);
\draw[fill,red] (180:0.5) circle (2pt);
\draw[fill,red] (270:0.5) circle (2pt);
\end{tikzpicture}&
&\begin{tikzpicture}[baseline=-.5ex]
\draw[thick] (0,0) circle (0.5);
\draw[fill,thick] (-1,-1) circle (2pt) -- (1,-1) circle (2pt);
\draw[thick] (-1,-1) to[out=90,in=135] (135:0.5);
\draw[fill] (135:0.5) circle (2pt);
\draw[thick] (1,-1) to[out=90,in=45] (45:0.5);
\draw[fill] (45:0.5) circle (2pt);
\draw[thick] (-1,-1) to[out=45,in=-135] (-135:0.5);
\draw[fill] (-135:0.5) circle (2pt);
\draw[thick] (1,-1) to[out=135,in=-45] (-45:0.5);
\draw[fill] (-45:0.5) circle (2pt);
\draw[thick,red] (0,0.5) -- (0,-0.5);
\draw[fill,red] (0,0.5) circle (2pt);
\draw[fill,red] (270:0.5) circle (2pt);
\end{tikzpicture}
\end{align*}
\caption{Some split angels}
\label{fig:split angels}
\end{figure}
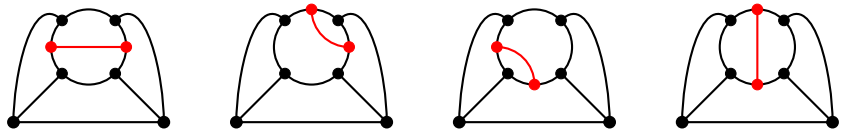

\begin{proof}[Proof of Lemma \ref{lem:triconnected case}]
We refer to Figure \ref{fig:split angels}. The first case expresses $b$ as the sum of loop cycles $b_1$ and $b_2$ to which the path argument applies---see Figure \ref{fig:splitting loop}. Note that this case includes the degenerate case where $x_1=y_1$ and $x_2=y_2$.

The second case expresses $b$ as the sum of loop cycles $b_1$ and $b_2$ such that the path argument applies to $b_2$ and the support of $b_1$ lies in an embedded angel graph with $x_1$ strictly closer to $y_1$. Similar remarks apply to the third case mutatis mutandis, while the fourth case expresses $b$ as the sum of two loop cycles, each of whose supports lies in an embedded angel graph with $x_i$ strictly closer to $y_i$ for both $i$. We now proceed by induction on the sum of the edge lengths of the arcs $(x_1,y_1)$ and $(x_2,y_2)$ in $b$, the base case being the degenerate case already considered.
\end{proof}

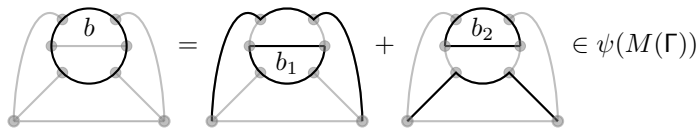
\begin{figure}
\[
\begin{tikzpicture}[baseline=-.5ex]
\begin{scope}[gray,opacity=0.5]
\draw[fill,thick] (-1,-1) circle (2pt) -- (1,-1) circle (2pt);
\draw[thick] (-1,-1) to[out=90,in=135] (135:0.5);
\draw[fill] (135:0.5) circle (2pt);
\draw[thick] (1,-1) to[out=90,in=45] (45:0.5);
\draw[fill] (45:0.5) circle (2pt);
\draw[thick] (-1,-1) to[out=45,in=-135] (-135:0.5);
\draw[fill] (-135:0.5) circle (2pt);
\draw[thick] (1,-1) to[out=135,in=-45] (-45:0.5);
\draw[fill] (-45:0.5) circle (2pt);
\draw[thick] (-0.5,0) -- (0.5,0);
\draw[fill] (0:0.5) circle (2pt) (180:0.5) circle (2pt);
\end{scope}
\draw[thick] (0,0) node[above] {$b$} circle (0.5);
\end{tikzpicture}=
\begin{tikzpicture}[baseline=-.5ex]
\begin{scope}[gray,opacity=0.5]
\draw[thick] (0,0) circle (0.5);
\draw[thick] (-1,-1) -- (1,-1);
\draw[thick] (-1,-1) to[out=90,in=135] (135:0.5);
\draw[fill] (135:0.5) circle (2pt);
\draw[thick] (1,-1) to[out=90,in=45] (45:0.5);
\draw[fill] (45:0.5) circle (2pt);
\draw[thick] (-1,-1) to[out=45,in=-135] (-135:0.5);
\draw[fill] (-135:0.5) circle (2pt);
\draw[thick] (1,-1) to[out=135,in=-45] (-45:0.5);
\draw[fill] (-45:0.5) circle (2pt);
\end{scope}
\draw[thick] (180:0.5) arc (180:360:0.5);
\draw[thick] (-0.5,0) -- (0.5,0);
\draw[fill,gray,opacity=0.5] (0:0.5) circle (2pt) (180:0.5) circle (2pt);
\draw[thick] (-1,-1) to[out=90,in=135] (135:0.5) arc (135:45:0.5) to[out=45,in=90] (1,-1);
\draw (-90:0.5) node[above] {$b_1$};
\draw[fill,thick,gray,opacity=0.5] (-1,-1) circle (2pt) (1,-1) circle (2pt);
\end{tikzpicture}+
\begin{tikzpicture}[baseline=-.5ex]
\begin{scope}[gray,opacity=0.5]
\draw[thick] (0,0) circle (0.5);
\draw[thick] (-1,-1) -- (1,-1);
\draw[thick] (-1,-1) to[out=90,in=135] (135:0.5);
\draw[fill,gray,opacity=0.5] (135:0.5) circle (2pt);
\draw[thick] (1,-1) to[out=90,in=45] (45:0.5);
\draw[fill,gray,opacity=0.5] (45:0.5) circle (2pt);
\draw[thick] (-1,-1) to[out=45,in=-135] (-135:0.5);
\draw[fill,gray,opacity=0.5] (-135:0.5) circle (2pt);
\draw[thick] (1,-1) to[out=135,in=-45] (-45:0.5);
\draw[fill,gray,opacity=0.5] (-45:0.5) circle (2pt);
\end{scope}
\draw[thick] (180:0.5) arc (180:0:0.5);
\draw[thick] (-0.5,0) -- (0.5,0);
\draw[fill,gray,opacity=0.5] (0:0.5) circle (2pt) (180:0.5) circle (2pt);
\draw[thick] (-1,-1) -- (-135:0.5) arc (-135:-45:0.5) -- (1,-1);
\draw (90:0.5) node[below] {$b_2$};
\draw[fill,thick,gray,opacity=0.5] (-1,-1) circle (2pt) (1,-1) circle (2pt);
\end{tikzpicture}\in\psi(M(\graf))
\]
\caption{Splitting a loop cycle}
\label{fig:splitting loop}
\end{figure}

We are now in a position to complete the proof of our base cases.

\begin{proposition}\label{prop:base case}
The conclusion of Theorem \ref{thm:reformulation} holds under the further assumption that $\graf_e$ is biconnected and $N_2(\graf_e)=1$.
\end{proposition}
\begin{proof}
If $\graf_e$ is a cycle graph or homeomorphic to $\thetagraph{n}$ for some $n\geq3$, then $\graf$ is among the examples considered in Section \ref{section:rogues gallery}, so we may assume by Theorem \ref{thm:decomposition theory} that $\graf_e$ is triconnected. 

Let $c=\sum_{j=1}^s q_jb_j$ be a pesky cycle. Each $b_j$ satisfies {\bf P} by definition, and {\bf S} implies that the support of $b_j$ does not intersect $\{v,v'\}$. Thus, Lemma \ref{lem:triconnected case} implies that $[b_j]\in \psi(M(\graf))$ for each $j$, whence $[c]\in \psi(M(\graf))$. Proposition \ref{prop:cokernel generators} now implies that $\ker(\delta)/\psi(M(\graf))=0$, as desired.
\end{proof}

\subsection{Angelic graphs}\label{section:angelic graphs}

In this section, we supply the missing ingredient in the proof of Lemma \ref{lem:triconnected case}. 

We employ uniform notation when dealing with angel graphs, some of which is indicated in Figure \ref{fig:angels}. The corner vertices are denoted $v$ and $v'$, and the edge connecting them is denoted $e$. The central loop is denoted $S$, and the two unique paths from $v$ to $v'$ avoiding $e$ and involving at most two vertices of $S$ are denoted $\gamma_1$ and $\gamma_2$, respectively. Considering $\gamma_i$ as oriented from $v$ to $v'$, the first point of intersection of $\gamma_i$ with $S$ is denoted $x_i$ and the last is denoted $y_i$ (it may be that $x_i=y_i$).

\begin{definition}\label{def:split angels}
A \emph{split angel} is an angel together an additional edge connecting two distinct components of $S\setminus \{x_1,y_1,x_2,y_2\}$.
\end{definition}

Some split angels are depicted in Figure \ref{fig:split angels}. 

\begin{definition} An \emph{angelic graph} is a topological embedding $\angelgraph\to \graf,$ where $\angelgraph$ is an angel. If the embedding extends to an embedding of a split angel, then the angelic graph is said to be \emph{split}.
\end{definition}

We abuse notation in referring to an angelic graph by the letter $\graf$, as well as by identifying parts of $\Omega$ with their images in $\graf$.

\begin{definition}
The angelic graphs $\Omega_1\to \graf$ and $\Omega_2\to \graf$ are called \emph{equivalent} if there are commuting diagrams of the following form:\[
\begin{tikzcd}
\relax[0,1]\arrow[d,"e"']\arrow[r,"e"]&\Omega_1\arrow[d]&& S^1\arrow[r,"S"]\arrow[d,"S"']&\Omega_1\arrow[d]\\
\Omega_2\arrow[r]&\graf&&\Omega_2\arrow[r]&\graf.
\end{tikzcd}\]
\end{definition}

\begin{proposition}\label{prop:angel}
Let $\graf$ be an angelic graph satisfying the following three conditions:
\begin{enumerate}
\item\label{planar} $\graf$ is planar 
\item\label{connected} $\graf\setminus e$ is triconnected, and
\item\label{P} $v$ and $v'$ lie in distinct components of $\graf \setminus (e\cup S)$.
\end{enumerate}
Then $\graf$ is split up to equivalence.
\end{proposition}

The key input to the proof will be the following result, which should be thought of as an infinitesimal version of Proposition \ref{prop:angel}.

\begin{lemma}\label{lem:essential vertices}
Let $\graf$ be an angelic graph satisfying the hypotheses of Proposition \ref{prop:angel}. Up to equivalence, each component of $S\setminus\{y_1,x_2\}$ contains an essential vertex.
\end{lemma}
\begin{proof}
If $x_1\neq y_1$ and $x_2\neq y_2$, then all four are essential, and there is nothing to show, so we may assume that $x_2=y_2$. Denote the components in question by $S'$ and $S''$. We may assume that $S'$ contains an essential vertex; indeed, in a minimal simplicial representative of $\graf$, $S$ contains some vertex other than $y_1$ and $x_2$, and assuming all such vertices to be bivalent implies that $\{y_1,x_2\}$ is a $2$-cut, contradicting (\ref{connected}). 

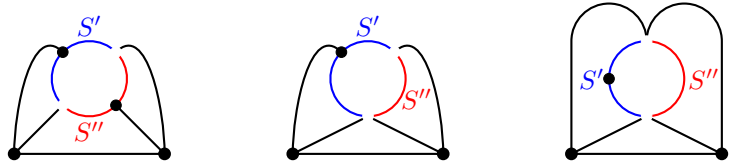
\begin{figure}[ht]
\begin{align*}
&\begin{tikzpicture}[baseline=-.5ex]
\draw[thick,blue] (45:0.5) arc (45:225:0.5) (90:0.7) node {$S'$};
\draw[thick,red] (225:0.5) arc (-135:45:0.5) (-90:0.7) node {$S''$};
\draw[fill,thick] (-1,-1) circle (2pt) -- (1,-1) circle (2pt);
\draw[thick] (-1,-1) to[out=90,in=135] (135:0.5);
\draw[fill] (135:0.5) circle (2pt);
\draw[thick] (1,-1) to[out=90,in=45] (45:0.5);
\draw[fill,white] (45:0.5) circle (2pt);
\draw[thick] (-1,-1) to[out=45,in=-135] (-135:0.5);
\draw[fill,white] (-135:0.5) circle (2pt);
\draw[thick] (1,-1) to[out=135,in=-45] (-45:0.5);
\draw[fill] (-45:0.5) circle (2pt);
\end{tikzpicture}&
&\begin{tikzpicture}[baseline=-.5ex]
\draw[thick,blue] (45:0.5) arc (45:270:0.5) (90:0.7) node {$S'$};
\draw[thick,red] (270:0.5) arc (-90:45:0.5) (-22.5:0.7) node {$S''$};
\draw[fill,thick] (-1,-1) circle (2pt) -- (1,-1) circle (2pt);
\draw[thick] (-1,-1) to[out=90,in=135] (135:0.5);
\draw[fill] (135:0.5) circle (2pt);
\draw[thick] (1,-1) to[out=90,in=45] (45:0.5);
\draw[fill,white] (45:0.5) circle (2pt);
\draw[thick] (-1,-1) -- (0,-0.5);
\draw[thick] (1,-1) -- (0,-0.5);
\draw[fill,white] (0,-0.5) circle (2pt);
\end{tikzpicture}&
&\begin{tikzpicture}[baseline=-.5ex]
\draw[thick,blue] (90:0.5) arc (90:270:0.5) (180:0.7) node {$S'$\,};
\draw[thick,red] (270:0.5) arc (-90:90:0.5) (0:0.7) node {\,\,$S''$};
\draw[fill] (-0.5,0) circle (2pt);
\draw[fill,thick] (-1,-1) circle (2pt) -- (1,-1) circle (2pt);
\draw[thick] (-1,-1) -- (-1,0.5) arc (180:0:0.5);
\draw[thick] (1,-1) -- (1,0.5) arc (0:180:0.5);
\draw[fill,white] (0,0.5) circle (2pt);
\draw[thick] (-1,-1) -- (0,-0.5);
\draw[thick] (1,-1) -- (0,-0.5);
\draw[fill,white] (0,-0.5) circle (2pt);
\end{tikzpicture}
\end{align*}
\caption{The connected components $S'$ and $S''$ of $S\setminus\{y_1,x_2\}$}
\label{fig:complement of angels}
\end{figure}

We may further assume that $S''$ contains no essential vertex, in which case (\ref{connected}) guarantees paths in $\graf$ connecting $S'$ to each component of the complement in $\Omega$ of the star of $\{y_1,x_2\}$, which furthermore avoid $e$, $y_1$, $x_2$, and $S''$. For the last, we have used that $S''$ is contained in the open star of $\{y_1,x_2\}$ in a minimal simplicial representative.

If $x_1=y_1$ and $x_2=y_2$, then there are two components and hence two paths. By (\ref{P}) may assume that these paths are disjoint away from $S$ and that neither intersects both $\gamma_1$ and $\gamma_2$ away from $S$. Using these paths, we obtain a topologically embedded $\completegraph{3,3}$, contradicting (\ref{planar})---see Figure \ref{fig:angel bipartite}. We take this case as the base case in an induction on the edge length of $\gamma_1\cap S$  as calculated in $\graf$ (in $\Omega$, this edge length is always either $0$ or $1$).

For the induction step, assume that $x_1\neq y_1$, so there is only one path $\gamma$. A similar contradiction of (\ref{planar}) as above is achieved unless both endpoints of $\gamma$ lie on $\gamma_1$. Replacing a segment of $\gamma_1$ with $\gamma$ now produces an equivalent angelic graph in which $\gamma_1\cap S$ has strictly shorter edge length---see Figure \ref{fig:induction step}.
\end{proof}

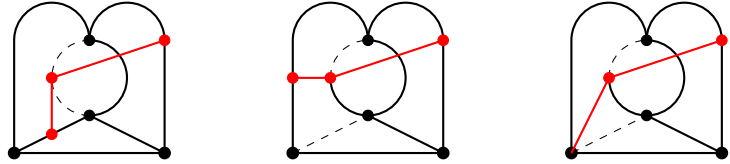
\begin{figure}[ht]
\begin{align*}&\begin{tikzpicture}[baseline=-.5ex,xscale=-1]
\draw[thick] (0,-0.5) arc (270:90:0.5);
\draw[dashed] (0,-0.5) arc (-90:90:0.5);
\draw[fill,thick] (-1,-1) circle (2pt) -- (1,-1) circle (2pt);
\draw[thick] (-1,-1) -- (-1,0.5) arc (180:0:0.5);
\draw[thick] (1,-1) -- (1,0.5) arc (0:180:0.5);
\draw[fill] (0,0.5) circle (2pt);
\draw[thick] (-1,-1) -- (0,-0.5);
\draw[thick] (1,-1) -- (0,-0.5);
\draw[fill] (0,-0.5) circle (2pt);
\draw[fill,red] (0.5, -0.75) circle (2pt);
\draw[fill,red] (-1,0.5) circle (2pt);
\draw[fill,red] (0.5,0) circle (2pt);
\draw[thick,red] (0.5,0) -- (0.5,-0.75);
\draw[thick,red] (0.5,0) -- (-1,0.5);
\end{tikzpicture}&
&\begin{tikzpicture}[baseline=-.5ex,xscale=-1]
\draw[thick] (0,-0.5) arc (270:90:0.5);
\draw[thick] (0,-0.5) arc (-90:0:0.5);
\draw[dashed] (0.5,0) arc (0:90:0.5);
\draw[fill,thick] (-1,-1) circle (2pt) -- (1,-1) circle (2pt);
\draw[thick] (-1,-1) -- (-1,0.5) arc (180:0:0.5);
\draw[thick] (1,-1) -- (1,0.5) arc (0:180:0.5);
\draw[fill] (0,0.5) circle (2pt);
\draw[thick] (-1,-1) -- (0,-0.5);
\draw[dashed] (1,-1) -- (0,-0.5);
\draw[fill] (0,-0.5) circle (2pt);
\draw[fill,red] (1, 0) circle (2pt);
\draw[fill,red] (-1,0.5) circle (2pt);
\draw[fill,red] (0.5,0) circle (2pt);
\draw[thick,red] (0.5,0) -- (1,0);
\draw[thick,red] (0.5,0) -- (-1,0.5);
\end{tikzpicture}&
&\begin{tikzpicture}[baseline=-.5ex,xscale=-1]
\draw[thick] (0,-0.5) arc (270:90:0.5);
\draw[thick] (0,-0.5) arc (-90:0:0.5);
\draw[dashed] (0.5,0) arc (0:90:0.5);
\draw[fill,thick] (-1,-1) circle (2pt) -- (1,-1) circle (2pt);
\draw[thick] (-1,-1) -- (-1,0.5) arc (180:0:0.5);
\draw[thick] (1,-1) -- (1,0.5) arc (0:180:0.5);
\draw[fill] (0,0.5) circle (2pt);
\draw[thick] (-1,-1) -- (0,-0.5);
\draw[dashed] (1,-1) -- (0,-0.5);
\draw[fill] (0,-0.5) circle (2pt);
\draw[fill,red] (-1,0.5) circle (2pt);
\draw[fill,red] (0.5,0) circle (2pt);
\draw[thick,red] (0.5,0) -- (1,-1);
\draw[thick,red] (0.5,0) -- (-1,0.5);
\end{tikzpicture}
\end{align*}
\caption{Embedded copies of $\completegraph{3,3}$}
\label{fig:angel bipartite}
\end{figure}

\begin{figure}[ht]
\begin{align*}
&\begin{tikzpicture}[baseline=-.5ex,xscale=-1]
\draw[thick] (90:0.5) arc (90:240:0.5);
\draw[thick] (-1,-1) to[out=90,in=135] (135:0.5);
\draw[thick] (1,-1) to[out=90,in=45] (45:0.5);
\draw[fill,white] (45:0.5) circle (5pt);
\draw[thick, red] (0.5,0.45) -- (0,0.5);
\draw[fill,red] (0.5,0.45) circle (2pt);
\draw[fill,red] (90:0.5) circle(2pt); 
\draw[fill] (240:0.5) circle (2pt);
\draw[fill] (135:0.5) circle (2pt);
\draw[thick,fill] (-1,-1) -- (-0.5,-0.75) circle (2pt);
\draw[thick,fill] (1,-1) -- (0.5,-0.75) circle (2pt);
\end{tikzpicture}&
&\begin{tikzpicture}[baseline=-.5ex,xscale=-1]
\draw[thick] (0,0) circle (0.5);
\draw[fill,thick] (-1,-1) circle (2pt) -- (1,-1) circle (2pt);
\draw[thick] (-1,-1) to[out=90,in=135] (135:0.5);
\draw[fill] (135:0.5) circle (2pt);
\draw[thick] (1,-1) to[out=90,in=0] (0.5,0.45) -- (0,0.5);
\draw[thick] (-1,-1) -- (0,-0.5);
\draw[thick] (1,-1) -- (0,-0.5);
\draw[fill] (0,0.5) circle (2pt);
\draw[fill] (0,-0.5) circle (2pt);
\end{tikzpicture}
\end{align*}
\caption{A path in red connecting the two components of the complement of the star of $\{y_1,x_2\}$ in $\Omega$, together with the resulting shortening of $\gamma_1\cap S$}
\label{fig:induction step}
\end{figure}
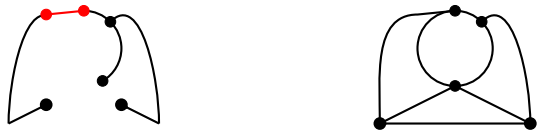

\begin{proof}[Proof of Proposition \ref{prop:angel}]
By Lemma \ref{lem:essential vertices}, we may assume that each component of $S\setminus\{y_1,x_2\}$ contains an essential vertex. Thus, by (\ref{connected}), there is a path connecting these components and avoiding $e$, $y_1$, and $x_2$. By (\ref{P}), we may assume that this path does not intersect both $\gamma_1$ and $\gamma_2$ away from $S$. There are now six possibilities, as depicted in Figure \ref{fig:splitting or reducing}. We conclude that, if $\graf$ is not split, then $\graf$ is equivalent to an angelic graph in which the edge length either of $\gamma_1\cap S$ or of $\gamma_2\cap S$ is strictly smaller, as in the proof of Lemma \ref{lem:essential vertices}. An induction on the sum of these edge lengths completes the proof.
\end{proof}

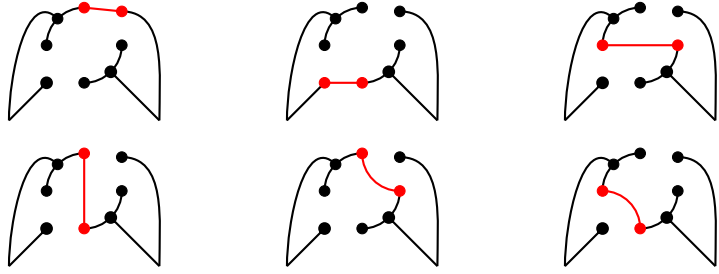
\begin{figure}[ht]
\begin{align*}
&\begin{tikzpicture}[baseline=-.5ex]
\draw[thick] (0,0.5) arc (90:180:0.5);
\draw[thick] (0:0.5) arc (0:-90:0.5);
\draw[thick] (-1,-1) to[out=90,in=135] (135:0.5);
\draw[fill] (135:0.5) circle (2pt);
\draw[thick] (1,-1) to[out=90,in=0] (0.5,0.45);
\draw[thick,red] (0.5,0.45) -- (0,0.5);
\draw[fill,red] (0.5,0.45) circle (2pt);
\draw[fill,thick] (-1,-1) -- (-0.5,-0.5) circle (2pt);
\draw[fill,thick] (1,-1) -- (-45:0.5) circle (2pt);
\draw[fill,red] (0,0.5) circle (2pt);
\draw[fill] (180:0.5) circle (2pt);
\draw[fill] (0.5,0) circle (2pt) (270:0.5) circle (2pt);
\end{tikzpicture}&&
\begin{tikzpicture}[baseline=-.5ex]
\draw[thick] (0,0.5) arc (90:180:0.5);
\draw[thick] (0:0.5) arc (0:-90:0.5);
\draw[thick] (-1,-1) to[out=90,in=135] (135:0.5);
\draw[thick] (1,-1) to[out=90,in=0] (0.5,0.45);
\draw[thick,red] (-0.5,-0.5) -- (0,-0.5);
\draw[fill] (0,0.5) circle (2pt) (180:0.5) circle (2pt);
\draw[fill] (0.5,0) circle (2pt);
\draw[fill,red] (270:0.5) circle (2pt);
\draw[fill] (135:0.5) circle (2pt);
\draw[fill] (0.5,0.45) circle (2pt);
\draw[fill,thick] (1,-1) -- (-45:0.5) circle (2pt);
\draw[fill,thick] (-1,-1) -- (-0.5,-0.5);
\draw[fill,red] (-.5,-.5) circle (2pt);
\end{tikzpicture}&&
\begin{tikzpicture}[baseline=-.5ex]
\draw[thick] (0,0.5) arc (90:180:0.5);
\draw[thick] (0:0.5) arc (0:-90:0.5);
\draw[thick,red] (-0.5,0) -- (0.5,0);
\draw[fill] (0,0.5) circle (2pt);
\draw[fill,red] (180:0.5) circle (2pt);
\draw[fill,red] (0.5,0) circle (2pt);
\draw[fill] (270:0.5) circle (2pt);
\draw[thick] (-1,-1) to[out=90,in=135] (135:0.5);
\draw[fill] (135:0.5) circle (2pt);
\draw[thick] (1,-1) to[out=90,in=0] (0.5,0.45);
\draw[fill] (0.5,0.45) circle (2pt);
\draw[fill,thick] (-1,-1) -- (-0.5,-0.5) circle (2pt);
\draw[fill,thick] (1,-1) -- (-45:0.5) circle (2pt);
\end{tikzpicture}\\
&\begin{tikzpicture}[baseline=-.5ex]
\draw[thick] (0,0.5) arc (90:180:0.5);
\draw[thick] (0:0.5) arc (0:-90:0.5);
\draw[thick,red] (0,0.5) -- (0,-0.5);
\draw[fill,red] (0,0.5) circle (2pt);
\draw[fill] (180:0.5) circle (2pt);
\draw[fill] (0.5,0) circle (2pt);
\draw[fill,red] (270:0.5) circle (2pt);
\draw[thick] (-1,-1) to[out=90,in=135] (135:0.5);
\draw[fill] (135:0.5) circle (2pt);
\draw[thick] (1,-1) to[out=90,in=0] (0.5,0.45);
\draw[fill] (0.5,0.45) circle (2pt);
\draw[fill,thick] (-1,-1) -- (-0.5,-0.5) circle (2pt);
\draw[fill,thick] (1,-1) -- (-45:0.5) circle (2pt);
\end{tikzpicture}&
&\begin{tikzpicture}[baseline=-.5ex]
\draw[thick] (0,0.5) arc (90:180:0.5);
\draw[thick] (0:0.5) arc (0:-90:0.5);
\draw[thick,red] (0, 0.5) arc (180:270:0.5);
\draw[fill,red] (0,0.5) circle (2pt);
\draw[fill] (180:0.5) circle (2pt);
\draw[fill,red] (0.5,0) circle (2pt);
\draw[fill] (270:0.5) circle (2pt);
\draw[thick] (-1,-1) to[out=90,in=135] (135:0.5);
\draw[fill] (135:0.5) circle (2pt);
\draw[thick] (1,-1) to[out=90,in=0] (0.5,0.45);
\draw[fill] (0.5,0.45) circle (2pt);
\draw[fill,thick] (-1,-1) -- (-0.5,-0.5) circle (2pt);
\draw[fill,thick] (1,-1) -- (-45:0.5) circle (2pt);
\end{tikzpicture}&&
\begin{tikzpicture}[baseline=-.5ex]
\draw[thick] (0,0.5) arc (90:180:0.5);
\draw[thick] (0:0.5) arc (0:-90:0.5);
\draw[thick,red] (0, -0.5) arc (0:90:0.5);
\draw[fill] (0,0.5) circle (2pt);
\draw[fill,red] (180:0.5) circle (2pt);
\draw[fill] (0.5,0) circle (2pt);
\draw[fill,red] (270:0.5) circle (2pt);
\draw[thick] (-1,-1) to[out=90,in=135] (135:0.5);
\draw[fill] (135:0.5) circle (2pt);
\draw[thick] (1,-1) to[out=90,in=0] (0.5,0.45);
\draw[fill] (0.5,0.45) circle (2pt);
\draw[fill,thick] (-1,-1) -- (-0.5,-0.5) circle (2pt);
\draw[fill,thick] (1,-1) -- (-45:0.5) circle (2pt);
\end{tikzpicture}
\end{align*}
\caption{Splitting or reducing arc length}
\label{fig:splitting or reducing}
\end{figure}

\section{Induction step}\label{section:induction}

In this section, we complete the proof of Theorem \ref{thm:reformulation}. Maintaining Notation \ref{notation:setup}, we assume that $\graf_e$ is biconnected and proceed by induction on $N_2(\graf_e)$, the base case being Proposition \ref{prop:base case}.

\subsection{Setup and reductions} We assume throughout this section and the next that $\graf_e$ has a $2$-cut $\{x,y\}$ with $\{x,y\}$-components $\{\graphfont{\Delta}_{i}\}_{i=1}^m$ such that $N_2(\overline{\graphfont{\Delta}}_{i})<N_2(\graf_e)$ for each $i$. Given a loop cycle $b$ in $\graf_w$ (whose support necessarily lies in $\graf_e$), we say that $b$ is \emph{local} if its support lies entirely in a single $\{x,y\}$-component; otherwise, we say that $b$ is \emph{global} (see Figure \ref{figure:simple necklace}). Note that the support of a global loop cycle is necessarily contained in the union of exactly two $\{x,y\}$-components.

\begin{figure}[ht]
\begin{align*}
&\begin{tikzpicture}[baseline=-0.5ex]
\fill(-2,1) circle (2pt) node[above] {$x$} (-2,-1) circle (2pt) node[below] {$y$};
\draw(-2,1)--(-3.2,0.5) (-2,1)--(-3,0.5) (-2,1)--(-2.8,0.5);
\draw(-2,-1)--(-3.2,-0.5) (-2,-1)--(-2.8,-0.5);
\draw(-2,1)--(-1.2,0.5) (-2,1)--(-0.8,0.5);
\draw(-1.5,-0.5) rectangle (-0.5,0.5) (-1,0) node {$\graphfont{\Delta}_2$};
\draw(-2,-1)--(-1.2,-0.5) (-2,-1)--(-1,-0.5) (-2,-1)--(-0.8,-0.5);
\draw (-3.5,-0.5) rectangle (-2.5,0.5) (-3,0) node {$\graphfont{\Delta}_1$};
\draw[red] (-3,0) circle (0.4) (-2.5,0) node[right] {$b$};
\end{tikzpicture}&
&\begin{tikzpicture}[baseline=-0.5ex]
\fill(-2,1) circle (2pt) node[above] {$x$} (-2,-1) circle (2pt) node[below] {$y$};
\draw(-2,1)--(-3.2,0.5) (-2,1)--(-3,0.5) (-2,1)--(-2.8,0.5);
\draw(-2,-1)--(-3.2,-0.5) (-2,-1)--(-2.8,-0.5);
\draw(-2,1)--(-1.2,0.5) (-2,1)--(-0.8,0.5);
\draw(-1.5,-0.5) rectangle (-0.5,0.5) (-1,0) node {$\graphfont{\Delta}_2$};
\draw(-2,-1)--(-1.2,-0.5) (-2,-1)--(-1,-0.5) (-2,-1)--(-0.8,-0.5);
\draw[red] (-2,1)--node[midway,above left] {$b'$} (-3.2,0.5)--(-3.2,-0.5)--(-2,-1)--(-0.8,-0.5)--(-0.8,0.5)--cycle;
\draw (-3.5,-0.5) rectangle (-2.5,0.5) (-3,0) node {$\graphfont{\Delta}_1$};
\end{tikzpicture}
\end{align*}
\caption{A local loop cycle $b$ and a global loop cycle $b'$}
\label{figure:simple necklace}
\end{figure}
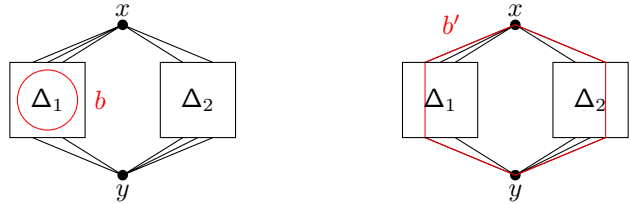

We begin by dealing with the local case.

\begin{lemma}\label{lem:local case}
Suppose that $\graf$ is planar. If $b$ is a local loop cycle in $\graf_w$ with support disjoint from $\{v,v'\}$, then $[b]\in \psi(M(\graf))$.
\end{lemma}
\begin{proof}
Without loss of generality, the support of $b$ lies in $\graphfont{\Delta}_1$. By inspection, we have $[b]=\iota_*\sigma_*([b])$, where $\sigma$ is the surgery along $\graf'=\bigcup_{i=2}^m\graphfont{\Delta}_i$ using the vertices $\{x,y\}$, and $\iota$ is any section.

If one of $v$ or $v'$, say $v$, is contained in $\Delta_1$, then we may regard the completion $\overline{\Delta}_1$ as $\Delta_e$, where $\Delta$ is the graph obtained by adding an edge between $v$ and a bivalent vertex $v'$ added to the edge $e_{xy}$. The claim now follows from Example \ref{example:completion surgery} and Proposition \ref{prop:base case} by induction on $N_2(\graf_e)$. Note that we use the assumption that $b$ is local to ensure that the loop cycle resulting from surgery again has support disjoint from $\{v,v'\}$, permitting the induction.

If neither $v$ nor $v'$ is contained in $\Delta_1$, then we may assume by the path argument that $v\in \Delta_2$ and $v'\in \Delta_3$. For the same reason, writing $S$ for the support of $b$, we may further assume that $\{x,y\}\subseteq S$. Since $\Delta_1\setminus\{x,y\}$ is connected, the two components of $S\setminus\{x,y\}$ are joined by a path in $\Delta_1$. Using this path, we may write $b=b_1+b_2$, where each $b_i$ is a loop cycle with support avoiding $\{v,v'\}$ and not containing $\{x,y\}$, hence yielding to the path argument.
\end{proof}

A similar tactic dispenses with global loop cycles in certain situations.

\begin{lemma}\label{lem:partial global case}
Suppose that $\graf$ is planar. If $b$ is a global loop cycle in $\graf_w$ with support disjoint from $\{v,v'\}$, and if $v$ and $v'$ lie in a common $\{x,y\}$-component, then $[b]\in \psi(M(\graf))$.
\end{lemma}
\begin{proof}
We begin by observing that $\{x,y\}\cap \{v,v'\}=\varnothing$ by our assumptions on $b$, so $v$ and $v'$ in fact lie in the same component of $\graf\setminus\{x,y\}$. If $b$ does not intersect this component, then the claim follows from the path argument; otherwise, the inductive surgery argument of Lemma \ref{lem:local case} applies. Here, we instead use the fact that $v$ and $v'$ lie in the same component of $\graf\setminus\{x,y\}$ to ensure that the loop cycle resulting from surgery again has support disjoint from $\{v,v'\}$.
\end{proof}

\subsection{Conclusion} We begin with a simple observation concerning star classes, which is a consequence of \cite[Lemma C.14]{AnDrummondColeKnudsen:SSGBG}.

\begin{lemma}\label{lem:independent stars}
Let $\{x,y\}$ be a $2$-cut in $\graphfont{\Delta}$ with $\{x,y\}$-components $\{\graphfont{\Delta}_i\}_{i=1}^m$ and $a$ and $a'$ star cycles in $\graphfont{\Delta}$. Write $i_j$ for the index such that the $j$th half-edge involved in $a$ lies in $\graphfont{\Delta}_{i_j}$ (resp. $i_j'$, $a'$). If $\{i_j\}_{j=1}^3\neq \{i_j'\}_{j=1}^3$, then $[a_1]$ and $[a_2]$ are independent in the sense that \[m[a_1]+n[a_2]=0\implies m[a_1]=0.\]
\end{lemma}

We now consider the following simple case.

\begin{lemma}\label{lem:not pesky}
Suppose that $\graf$ is planar and the $2$-cut $\{x,y\}$ is disjoint from and separates $v$ and $v'$. Let $b$ be a global loop cycle, $e_0$ an edge lying in the support of $b$, and $m$ and $n$ non-negative integers. If the cycle $ne_0^mb$ satisfies {\bf K}, then $n=0$.
\end{lemma}
\begin{proof}
Suppose that $ne_0^mb$ satisfies {\bf K}. Without loss of generality, the support of $b$ lies in $\graphfont{\Delta}_1\cup \graphfont{\Delta}_2$. By {\bf K} and the Q-relation, we have \[0=ne_0^m(e-e')[b]=\pm ne_0^m([a]-[a']),\] where $a$ is a star cycle satisfying the following conditions with respect to $v$ (resp. $a'$, $v'$):
\begin{enumerate}
\item if $v$ lies in $\graphfont{\Delta}_i$ for some $i\in \{1,2\}$, then $a$ involves at least two half-edges lying in $\graphfont{\Delta}_i$ (resp. $v'$, $a'$);
\item if $v$ does not lie in $\graphfont{\Delta}_i$ for any $i\in \{1,2\}$, then $a$ involves a single half-edge in each of $\graphfont{\Delta}_1$, $\graphfont{\Delta}_2$, and the $\{x,y\}$-component of $v$.
\end{enumerate}
These conditions guarantee that the hypothesis of Lemma \ref{lem:independent stars} obtains. Since multiplication by $e_0$ is injective by \cite[Prop. 5.21]{AnDrummondColeKnudsen:SSGBG}, it follows that $n$ annihilates $[a]-[a']$ and hence each class separately. We conclude that $n=0$ by planarity and \cite[Cor. 3.6]{KoPark:CGBG}.
\end{proof}

Note that Lemma \ref{lem:not pesky} does not require the support of $b$ to be disjoint from $\{v,v'\}$.

\begin{proof}[Proof of Theorem \ref{thm:reformulation}]
By Lemma \ref{lem:reduce to biconnected} and Proposition \ref{prop:base case}, it suffices to establish the conclusion under the further assumption that $\graf_e$ is biconnected with a $2$-cut $\{x,y\}$ such that $N_2(\overline{\graphfont{\Delta}}_{i})<N_2(\graf_e)$ for each $i$. Consider the pesky cycle $c=\sum_{j=1}^s q_jb_j$, where each $b_j$ is a loop cycle. By biconnectivity and Proposition \ref{prop:cokernel generators}, it suffices to show that $[c]\in \psi(M(\graf))$. By Lemmas \ref{lem:local case} and \ref{lem:partial global case}, we may further assume that each $b_j$ is global and that $\{x,y\}$ separates $v$ and $v'$. 

Unless $c=0$, since $b_1$ is global satisfying {\bf S}, its support intersects $\graphfont{\Delta}_1$ (up to symmetry) in a path $\gamma$ from $x$ to $y$ avoiding $v$ and $v'$. We call a loop cycle \emph{special} if the intersection of its support with $\graphfont{\Delta}_1$ is also $\gamma$. The support of an arbitrary global loop cycle $b$ intersects $\graphfont{\Delta}_1$ either in a path from $x$ to $y$ or in the set $\{x,y\}$. A cycle of the former type differs from a special loop cycle by a sum of local loop cycles, while a cycle of the latter type is the difference of two special loop cycles. Therefore, at the cost of introducing stabilized local loop cycles, which lie in $\psi(M(\graf))$ by Lemma \ref{lem:local case}, each $b_j$ may be assumed special. Moreover, fixing an edge $e_0$ lying in the path $\gamma$, we may take $q_j=e_0^{k-1}$ up to homology for each $j$. (and up to non-rigid star cycles which can be ignored by Lemma~\ref{lem:non-rigid in image}).

Let $\sigma$ denote surgery along $\graf'=\cup_{i=2}^m\graphfont{\Delta}_i$ using the vertices $\{x,y\}$. By Proposition \ref{prop:basic surgery}, there is a section $\iota$ such that $(1-\iota_*\sigma_*)([c])$ is a sum of a stabilized star classes, which lie in $\psi(M(\graf))$ by Lemma~\ref{lem:non-rigid in image} (these star classes are necessarily non-rigid by biconnectivity).
By inspection, since each $b_j$ is special, we have \[\iota_*\sigma_*([c])=\sum_{j=1}^se_0^{k-1}\iota_*\sigma_*([b_j])=ne_0^{k-1}[b],\] where $b$ is the loop cycle with support given by the union of $\gamma$ and $\iota(e_{xy})$ (note that we may have $s\neq n$, since each $b_j$ carries an orientation). By $\mathbb{Z}[E]$-linearity and the fact that $\ker(\delta)$ is $(e-e')$-torsion, we conclude that the cycle $ne_0^{k-1}b$ satisfies {\bf K}, whence $n=0$ by Lemma \ref{lem:not pesky}. It follows that $[c]\in \psi(M(\graf))$, as claimed.
\end{proof}

\bibliographystyle{amsalpha}
\bibliography{references}

\end{document}